\documentclass[12pt,oneside,reqno]{article}

\usepackage{amsmath,amssymb,amsthm,textcomp,mathtools}
\usepackage{amsfonts,graphicx}
\usepackage[mathscr]{eucal}
\usepackage{color}
\usepackage{csquotes}
\usepackage{enumitem}
\usepackage{authblk} 
\usepackage{comment}

\numberwithin{equation}{section}

\theoremstyle{definition}

\numberwithin{equation}{section}

\newtheorem{theorem}{\bf Theorem}[section]
\newtheorem{remark}{\bf Remark}[section]
\newtheorem{proposition}{Proposition}[section]
\newtheorem{lemma}{Lemma}[section]
\newtheorem{corollary}{Corollary}[section]

\newtheoremstyle
{remarkstyle}
{}
{11pt}
{}
{}
{\bfseries}
{:}
{     }
{\thmname{#1} \thmnumber{#2} }

\theoremstyle{remarkstyle}

\begin{document}
	\title{
    Patterned matrices with 
    random walk entries} 
\author{Arup Bose\thanks{bosearu@gmail.com \\ 
\ Research  of  Arup Bose is supported by the J.C.~Bose National Fellowship, JBR/2023/000023 from Anusandhan National Research Foundation, 
        Govt.~of India}}
        
\author{Pradeep Vishwakarma\thanks{vishwakarmapr.rs@gmail.com\\ 
\ Research  of  Pradeep Vishwakarma is supported by the National Post Doctoral Fellowship, PDF/2025/000076 from Anusandhan National Research Foundation, 
        Govt.~of India}}
\affil{Theoretical Statistics and Mathematics Unit, Indian Statistical Institute,\newline
Kolkata, 700108, West Bengal, India}
	\date{\today}	
	\maketitle
	\begin{abstract} It is well known that the weak limit of a suitably scaled continuous-time random walk (CTRW) is the Brownian motion. We investigate the convergence of certain patterned random matrices whose entries are independent CTRWs and their time-changed versions, in a non-commutative probability framework. For the Wigner link function,  the limits are free Brownian motion and its time-changed version driven by an inverse stable subordinator. For the symmetric circulant and the circulant with CTRW entries, we use their explicit eigenvalue expressions to define some empirical processes that converge weakly to a Brownian motion and a complex Brownian motion, respectively. For matrices with iid entries, and for elliptic matrices, the algebraic limits are equal in $*$-distribution to processes whose marginals are circular and elliptic variables, respectively.  A random time-changed variant of these results is also established. 
     \vspace{0.15cm} 
	\end{abstract}

\section{Introduction} 

Let $\{X_k\}_{k\geq 1}$ be independent and identically distributed (iid) random variables (rvs) with mean zero and variance one, and let 
$\{S_n:=\sum_{k=1}^{n}X_k\}_{n\ge0}$ with $S_0=0$. It is well known that the scaled  continuous-time random walk (CTRW),
$\{n^{-1/2}S_{[nt]}\}_{t\ge0}$, $n\ge1$,  where $[\cdot]$ is the greatest integer function, converges weakly to a (standard) Brownian motion (BM) $\{W(t)\}_{t\ge0}$, as $n\rightarrow\infty$. This conclusion remains true if $\{X_k\}$ are not necessarily identically distributed but are independent with uniformly bounded moments of all orders. The BM is a Gaussian process with independent and stationary increments, $W(0)=0$ and for all $0\leq s \leq t < \infty$, $W(t)-W(s)$ is a Gaussian variable with mean zero and variance $t-s$. 
 
Let $\{N(t)\}_{t\ge0}$ be a renewal process with iid inter-event times $\{J_k\}_{k\ge1}$, independent of $\{X_k\}_{k\ge1}$. Then $\{S_{N(t)}\}_{t\ge0}$
is a time-changed random walk. If $\{N(t)\}_{t\ge0}$  is a Poisson process with rate one and compounding rvs $\{X_k\}_{k\ge1}$ are standard Gaussian then the scaled CTRW, $\{n^{-1/2}S_{N(nt)}\}_{t\ge0}$, $n\ge1$, converges weakly in  appropriate Skorokhod topology to a BM (see \cite{Silvestrov2004}). Moreover, if $\{J_k\}_{k\ge1}$ belong to the domain of attraction of a stable law with index $\alpha\in(0,1)$ and $\{X_k\}_{k\ge1}$ belong to the domain of attraction of a standard Gaussian law, then the scaled CTRW, $\{n^{-\alpha/2}S_{N(nt)}\}_{t\ge0}$, $n\ge1$, $0<\alpha<1$, weakly converges in an appropriate Skorokhod topology to a time-changed Brownian motion (see \cite{Meerschaert2012, Silvestrov2004}), where the time changing component is an inverse stable subordinator (for definition see Section \ref{fpp}). 

In non-commutative probability, the analogues of the Gaussian distribution and the Brownian motion are respectively the semi-circle distribution (also known as the free Gaussian distribution) and the (standard) free Brownian motion (FBM), say $\{B(t)\}_{t \geq 0}$ (see Section \ref{fbm}) for which the increments are freely independent, $B(t)=0$, and the variable $B(t)-B(s)$ has a semi-circle distribution with variance $t-s$ for every $t \geq s$. At the same time, it is known that the algebraic limit variable of scaled Wigner matrices of increasing dimensions has the semi-circle distribution (see \cite{Wigner1958}), and that independent Wigner matrices are almost surely (a.s.) asymptotically free (see \cite{Anderson2009}, \cite{Bose2021}). 

Symmetric matrix valued random processes and their associated eigenvalue processes have been studied in great detail for many years since their formal introduction in \cite{Dyson1962}. The connection between positive semi-definite matrix valued random processes and the dynamical systems associated with their eigenvalues is discussed in \cite{Norris1986}. For more details on the recent development for the eigenvalues of matrix-valued random processes, we refer to \cite{Song2022, Song2023, Song2024}, and references therein. In  \cite{Bian1997}, FBM is obtained as a scaling limit of Wigner matrices with  independent Brownian motion entries. In the last few decades, the connection between symmetric matrix valued random processes and multivariate statistical analysis and their application have been investigated in various fields, such as financial data analysis (see \cite{Gouri2006, Gnoatto2012, Gnoatto2014}), computer vision (see \cite{Li2016}), and machine learning (see \cite{Zhang2006}).

Let $(\Omega,\mathcal{F},\mathbb{P})$ be a classical probability space. Let $\mathcal{A}_{\mathbb{P}}=\cap_{1\leq p<\infty}L^p(\mathbb{P})$ be the algebra of rvs with all moments finite and let $\mathcal{M}_n(\mathcal{A}_{\mathbb{P}})$ be the algebra of $n\times n$ complex random matrices with entries from $\mathcal{A}_{\mathbb{P}}$. Then, the pair $(\mathcal{M}_n(\mathcal{A}_{\mathbb{P}}), \tau_n)$ with $\tau_n=n^{-1}\mathbb{E}\mathrm{Trace}(\cdot)$ ($\mathbb{E}$ denotes expectation with respect to the probability $\mathbb{P}$) or $\tau_n=n^{-1}\mathrm{Trace}(\cdot)$ forms a non-commutative probability space (see Section \ref{NCP} for a formal definition). Let $\{A_n(t)\}_{t\ge0}$ be a collection from $\mathcal{M}_n(\mathcal{A}_{\mathbb{P}})$. Then, it is said to be algebraically (or jointly) convergent to a collection $\{A(t)\}_{t\ge0}$ if, for any $p\in\mathbb{N}$ and distinct $t_1,\dots,t_p\in[0,\infty)$, the matrices $A_n(t_1),\dots, A_n(t_p)$ jointly converge to variables $A(t_1),\dots, A(t_p)$ (see Section \ref{jointcov} for details).
 
Motivated by the above, we consider some sequences of patterned random matrices whose entries evolve as independent CTRWs. Let $\{X_{i,j,k}\}$ and  $\{X_{i,k}\}$ be independent random variables with mean $0$, variance $1$ and all moments uniformly bounded. For $i, j \geq 1$, let $S_{i,j,n}(t)=\sum_{k=1}^{[nt]}X_{i,j,k}$ and  $S_{i,n}(t)=\sum^{[nt]}_{k=1}X_{i,k}$ ,$n\ge1$ be independent random walks. 

We construct patterned random matrices with these random walks as entries as follows. For either $d=1$ or $d=2$, let $L^d_n:\{1,2,\dots,n\}^2\rightarrow\mathbb{Z}^d$ (called link functions).
By abuse of notation, we denote these functions by $L^d$. This link is called symmetric if $L^d(i, j)=L^d(j, i)$ for all $i, j$. 
We consider matrices \{$A_n(t)=((S_{L^d(i,j),[nt]}))_{1\leq i,j\leq n}\}_{t\ge0}$ in the NCP $(\mathcal{M}_n(\mathcal{A}_{\mathbb{P}}), \tau_n)$. These matrices use $\{S_{i,n}\}$ or $\{S_{i,j,n}\}$ depending on whether $d=1$ or $d=2$. 
We also consider a time-changed version of these matrices (details given later). These matrices are symmetric if $L^d$ is symmetric. 

After recalling in Section \ref{pre}, some preliminaries that we need, in Sections \ref{marginal} and \ref{joint}, we establish respectively, the marginal and joint convergence as defined above (that is, convergence of the expected average trace of any polynomial) of 
$\{n^{-1}A_n(t)\}_{t\ge0}$ for suitable \textit{symmetric} links $L$.  These algebraic convergence results also yield the a.s.~convergence of the empirical spectral distribution (ESD) of symmetric polynomials in $n^{-1}A_n(t), t_1 < \cdots < t_k$ for every $\{t_i\}$ (and even independent copies of these).

For the choice of the Wigner link $L^2(i,j)=(\max(i,j), \min (i,j))$, $\{n^{-1}A_n(t)\}_{t\ge0}$ converges algebraically to a free Brownian motion (FBM) $\{B(t)\}_{t\ge0}$. The ESD of any symmetric polynomial in these matrices converges a.s.~to the distributions determined by the moments of the corresponding (self-adjoint) polynomial of $\{B(t)\}_{t\ge0}$. 

The choice of the link $L^1(i,j)=n/2-|n/2-|i-j||$, $1\leq i,j\leq n$,  yields a Symmetric Circulant matrix. In this case, the matrices $\{n^{-1}A_n(t)\}_{t\ge0}$ are commutative, and converge algebraically to $\{W(t)\}_{t\ge0}$, whose finite dimensional moments agree with those of a BM. The ESD of any symmetric polynomial in these matrices converges to the distribution determined by the moments of the corresponding (self-adjoint) polynomial of $\{W(t)\}_{t\ge0}$. 

In Section \ref{symcircsbm}, with the above link function, we also consider the empirical process, $Y_n(t)=n^{-1}\sum_{r=1}^n\lambda_{r,n}(t)\mathbb{I}(U_n=r)$ where $\lambda_{r,n}(t), 1\leq r\leq n$ are the eigenvalues of $n^{-1}A_n(t)$ and $U_n$ is a uniform random variable, taking values in $\{1, \ldots, n\}$, independent of all other variables. We show that $Y_n(\cdot)$ converges weakly to a BM in the Skorokhod space $D[0, \infty)$.

In Section \ref{fbm} we introduce the time-changed FBM. In Section \ref{rmtstoppedrw}, we consider patterned random matrices with stopped (via fractional Poisson process) random walk entries as $\{A_n^\alpha(t)=((S_{L^d(i,j), N^\alpha(t)}))_{1\leq i,j\leq n}\}_{t\ge0}$, $0<\alpha\leq 1$, where the time-changing component $\{N^\alpha(t)\}_{t\ge0}$ is an independent fractional Poisson process (see Section \ref{fpp} for the definition). We establish the 
algebraic convergence of the matrices $n^{-(1+\alpha)/2}A_n^\alpha (t)_{t \geq 0}$, for suitable symmetric links. 

For the link function  $L^2(i,j)=(\max(i,j), \min (i,j))$, 
$\{n^{-(1+\alpha)/2)}A_n^\alpha(t)\}_{t\ge0}$ converges algebraically to a time changed FBM $\{B(L^\alpha(t))\}_{t\ge0}$, 
where $\{L^\alpha(t)\}_{t\ge0}$ is an inverse stable subordinator independent of $B(\cdot)$. For any $t$, 
 the moments of $B(L^\alpha(t))$ identify a unique probability distribution. Simulations (see Figure \ref{fig1}) show that its density function has a longer tail compared to the semi-circle density with variance $t$, and as $\alpha\to 1$, this density approaches the latter. As before, from the algebraic convergence, the a.s.~convergence of the ESD of any symmetric polynomial can also be concluded. 

Similarly, for the link $L^1(i,j)=n/2-|n/2-|i-j||$,
$\{n^{-(1+\alpha)/2)}A_n^\alpha(t)\}_{t\ge0}$ converges algebraically to a time changed BM $\{W(L^\alpha(t))\}_{t\ge0}$, 
where $\{L^\alpha(t)\}_{t\ge0}$ is an inverse stable subordinator independent of $W$. For any $t$, 
 the moments of $W(L^\alpha(t))$ identify a unique probability distribution. 
As before, from the algebraic convergence, the a.s.~convergence of the ESD of any symmetric polynomial can also be concluded.

In Section \ref{symcirctimechangedsbm} we consider the link function $L(i,j)=n/2-|n/2-|i-j||$, $1\leq i,j\leq n$, and a process $Y_n^\alpha(t)=n^{-1}\sum_{r=1}^{n-1}\lambda_{r,n}^\alpha(t)\mathbb{I}(U_n=r)$, $t\ge0$, where $\lambda_{r,n}(t)$ are the eigenvalues of $\{n^{-(1+\alpha)/2)}A_n^\alpha(t)\}_{t\ge0}$. Analogous to the convergence results in Section \ref{symcircsbm}, we show that this process converges weakly to $\{W(L^\alpha(t))\}_{t \ge 0}$ in Skorokhod topology, 
where $\{L^\alpha(t)\}_{t \ge 0}$ is an inverse stable subordinator that is independent of the Brownian motion $\{W(t)\}_{t\ge0}$.
  
In Section \ref{nonsymmetric} we deal with a few non-symmetric cases. The circulant, IID, and elliptic matrices produce algebraic limits as complex Brownian motion, free circulation motion,  and free elliptic motion. We also discuss their time changed versions. 

\section{Notations and preliminaries}\label{pre} In this section, we fix our notation and recall known concepts and results that will be used later. Let $\mathbb{C}$, $\mathbb{Z}$ and $\mathbb{N}$ denote the sets of complex numbers, integers and positive integers, respectively. For any $B\subset \mathbb{Z}$, $\#B$ will denote its cardinality. We denote constants by $K$ or $K_1,K_2,\dots$, often without specifying their dependence on different variables. 
	
\subsection{Non-commutative probability and (free) independence}\label{NCP} A non-commutative probability space (NCP) is a pair $(\mathcal{A},\tau)$ with unital algebra $\mathcal{A}$  over $\mathbb{C}$ and a linear functional $\tau:\mathcal{A}\rightarrow\mathbb{C}$ such that $\tau(1_{\mathcal{A}})=1$, where $1_{\mathcal{A}}$ denotes the identity in $\mathcal{A}$. The state $\tau$ is called tracial if $\tau(ab)=\tau(ba)$ for all $a,b$ in $\mathcal{A}$. If $\mathcal{A}$ is a $*$-algebra, that is, there is an involution or $*$-operator $*:\mathcal{A}\rightarrow\mathcal{A}$, then a state $\tau$ on $\mathcal{A}$ is called positive if $\tau(a^*a)\ge0$ for all $a\in\mathcal{A}$. The pair $(\mathcal{A},\tau)$ with $*$-algebra $\mathcal{A}$ and positive state $\tau$ is called $*$-probability space. 
		
	
\subsubsection{Independence and free independence} Let $(\mathcal{A},\tau)$ be an NCP and $\{\mathcal{A}_j\}_{j\in\mathcal{J}}$ be a collection of unital sub-algebras of $\mathcal{A}$. Then, they are called independent if they commute and for all finite subsets $\mathcal{I}\subset\mathcal{J}$ and $a_i\in\mathcal{A}_i$, $i\in\mathcal{I}$, the state $\tau$ factorizes as follows: $\tau\big(\prod_{i\in\mathcal{I}}a_i\big)=\prod_{i\in\mathcal{I}}\tau(a_i)$.
      
The sub-algebras $\{\mathcal{A}_j\}_{j\in\mathcal{J}}$ are free independent if for $k=1,2,\dots,n$, $n\in\mathbb{N}$ and $a_{k}\in\mathcal{A}_{i_k}$, $i_k\in\mathcal{I}$ such that $i_{k'}\ne i_{k'+1}$ for $k'=1,2,\dots,n-1$ and $\tau(a_k)=0$ for all $k$, we have $\tau(a_1\dots a_n)=0$.

Elements of an algebra are called independent (resp. free independent) if the sub-algebras generated by them are independent (resp. free independent). Appropriate extensions of these concepts are available for $*$-probability spaces.

\subsubsection{Joint distribution and asymptotic freeness} \label{jointcov} Let $(\mathcal{A},\tau)$ be an NCP. Let $\mathbb{C}\langle x_1,\dots, x_p\rangle$ be the set of all polynomials in non-commutative indeterminates $x_1,\dots, x_p$, and let $g$ denote a typical polynomial in these variables. For $a_1,\dots,a_p \in \mathcal{A}$, $p\ge1$, their joint distribution is a linear functional $\varphi:\mathbb{C}\langle a_1,\dots,a_p\rangle\to\mathbb{C}$ given as follows:
\begin{equation*}
    \varphi(g)= \tau(g(a_1,\dots,a_p)).
\end{equation*} 

Let $(\mathcal{A}_n,\tau_n)$, $n\ge1$ be NCPs. Let $\{a_{1,n},\dots,a_{p,n}\}\subset\mathcal{A}_n$, with joint distribution $\varphi_n$. Then, $a_{1,n},\dots, a_{p,n}$ are said to jointly converge (in the algebraic sense) if 
$\lim_{n\longrightarrow\infty}\varphi_n(g)$ exists for all fixed $g\in\mathbb{C}\langle a_{1,n}\dots, a_{p,n}\rangle$.
The limits can be used in a natural way to define an NCP, say $(\mathcal{A}, \varphi)$ where $\mathcal{A}$ is generated by indeterminate variables $a_1, \ldots, a_p$, so that $\lim_{n\longrightarrow\infty}\varphi_n(g)=\varphi(g)$ for all $g$. If the limiting variables $a_1,\dots,a_p$ are free independent then $a_{1,n},\dots,a_{p,n}$ are called asymptotically free. Again, obvious extensions of this concept is available for $*$-probability spaces.
        
Since we will be dealing with matrices, our $*$-algebra of interest will be $\mathcal{M}_n(\mathcal{A}_{\mathbb{P}})$ consisting of all  $n\times n$ complex random matrices with entries from $\mathcal{A}_{\mathbb{P}}$, where  $\mathcal{A}_{\mathbb{P}}=\cap_{1\leq p<\infty}L^p(\mathbb{P})$ is the algebra of all rvs defined on a given probability space with all moments finite. The states on $\mathcal{M}_n(\mathcal{A}_{\mathbb{P}})$ will be either 
$\tau_n=n^{-1}\mathbb{E}^{\mathbb{P}}\mathrm{Trace}(\cdot)$ or $\tau_n=n^{-1}\mathrm{Trace}(\cdot)$.
    
    \subsection{Inverse stable subordinator and fractional Poisson process}\label{fpp} For $0<\alpha<1$, a non-negative and non-decreasing L\'evy process $\{H^\alpha(t)\}_{t\ge0}$ is an $\alpha$-stable subordinator if its Laplace transform is given by $\mathbb{E}e^{-uH^\alpha(t)}=e^{-tu^\alpha}$, $u>0$ (see \cite{Meerscheart2013}). A first passage time process $\{L^\alpha(t)\}_{t\ge0}$ defined by $L^\alpha(t)\coloneqq\inf\{s\ge0:H^\alpha(s)\ge t\}$ is called inverse $\alpha$-stable subordinator. Also, we set  $L^1(t)=t$. The $l$th order moment of $L^\alpha(t)$ is given by
	\begin{equation}\label{lmomis}
		\mathbb{E}(L^\alpha(t))^l=\frac{t^{l\alpha}\Gamma(l+1)}{\Gamma(l\alpha+1)},\ l\ge1,\ t\ge0.
	\end{equation}
	
Let $\{N(t)\}_{t\ge0}$ be a homogeneous Poisson process with parameter $\lambda>0$, independent of $\{L^{\alpha}(t)\}_{t\ge0}$. Let us consider the following time-changed process:
	\begin{equation}\label{tcr}
		N^\alpha(t)\coloneqq N(L^{\alpha}(t)),\ 0<\alpha\leq1,\ t\ge0.
	\end{equation}
    The process $\{N^\alpha(t)\}_{t\ge0}$ is called the fractional Poisson process (fPP). It was introduced and studied in \cite{Meerschaert2011}. For other characterizations of the fPP, see \cite{Beghin2009, Laskin2003, Mainardi2004}.

    \subsection{Free Brownian motion (FBM) and time-change}\label{fbm}
      A \textit{free Brownian motion} (FBM) is a collection $\{B(t)\}_{t\ge0}$ of self-adjoint operators in a $*$-probability space $(\mathcal{A},\tau)$ (where $\mathcal{A}$ is generally a von-Neumann algebra) such that (for details see \cite{Biane1998, Bian1997}), \vskip3pt
        
    \noindent (i)  $B(0)=0$;\\
    \noindent (ii) For $p\in\mathbb{N}$ and $0\leq t_0<t_1<\dots<t_p<\infty$, the increments $B(t_1)-B(t_0), B(t_2)-B(t_1),\dots,B(t_p)-B(t_{p-1})$ are freely independent;\\
    \noindent (iii) For any $s,t\ge0$, 
    $B(s+t)-B(s)$ is a semi-circle variable with variance $t$. That is, its moments uniquely determine a semi-circular law with density given as follows:
    \begin{equation*}
			f(x,t)=\begin{cases}
				\displaystyle\frac{1}{2\pi t}\sqrt{4t-x^2},\ x\in[-2\sqrt{t}, 2\sqrt{t}],\\
				0,\ \text{otherwise}.
			\end{cases}
		\end{equation*} 

A Brownian motion time-changed by an independent inverse stable subordinator is a key mathematical model for analyzing anomalous diffusion in statistical mechanics. It is used for the modeling of the motion of a particle experiencing a trapping effect. 

Here, we define a random time-changed FBM driven by an inverse stable subordinator. Let $\{B(t)\}_{t\ge0}$ be an FBM associated to the NCP $(\mathcal{A},\tau)$. Let  $L^\alpha(t), t \geq 0$ be an inverse stable subordinator on $(\Omega, \mathcal{F}, \mathbb{P})$, independent of the FBM. From \cite{Veillette2010}, all  moments of $L^\alpha(t), t \geq 0$ are finite, 
Hence  $\mathcal{A}_\infty$, the set of all polynomials in $L^\alpha(t), t \geq 0$ is an NCP with its state as expectation operator $\mathbb{E}$. 
Let us consider a collection  $\{B^\alpha(t)\}_{t\ge0}$ of variables of the tensor product space $(\hat{\mathcal{A}}:=\mathcal{A}\otimes \mathcal{A}_\infty, \hat\tau:=\mathbb{E}\tau)$ 
 defined as follows:
\begin{align}
	B^\alpha(t)\coloneqq & B(L^\alpha(t)),\ 0<\alpha\leq1, \label{tcfbm}\\
    \hat{\tau}(B^\alpha(t_1)\dots B^\alpha(t_l))\coloneqq & \mathbb{E}\tau(B^\alpha(t_1)\dots B^\alpha(t_l))\nonumber
    \end{align}
    for all $l\ge1$ and $t_1,\dots,t_l\ge0$. Equivalently,
    \begin{equation*}
	\hat{\tau}(B^\alpha(t_1)\dots B^\alpha(t_l))= \int_{\mathbb{R}^l_+}\tau(B(x_1)\dots B(x_l))\mathrm{Pr}\{L^\alpha(t_1)\in\mathrm{d}x_1,\dots,L^\alpha(t_l)\in\mathrm{d}x_l\}.
\end{equation*}
$\{B^\alpha(t)\}_{t\ge0}$ will be called a time-changed FBM. For $\alpha=1$, it reduces to the FBM.

 From \cite{Veillette2010}, it follows that all the finite dimensional joint moments of the inverse stable subordinator exist. Therefore, the linear functional $\hat{\tau}$ is well defined. In particular, for $l\ge1$, we have $\hat{\tau} ((B^\alpha(t))^l)=0$ for all $t\ge0$ whenever $l$ is odd and
\begin{equation}\label{evnmom}
	\hat{\tau}((B^\alpha(t))^{2m})=C_{m}\mathbb{E}(L^\alpha(t))^m=C_{m}\frac{t^{m\alpha}\Gamma(m+1)}{\Gamma(m\alpha+1)},\ m\ge1,\ t\ge0,
\end{equation}
where $C_{m}=(2m)!/(m+1)!m!$ is the $m$th Catalan number.  

The following result identifies the unique distribution whose odd moments are zero and even moments are given by (\ref{evnmom}).  
\begin{proposition}
	Let $Z(t)$ be a semi-circle rv with mean zero and variance $t>0$ that is independent of an inverse $\alpha$-stable subordinator $\{L^\alpha(t)\}_{t\ge0}$. Then, 
	\begin{equation}\label{tcmgf}
		\mathbb{E}e^{uZ(L^\alpha(t))}=(u^2t^\alpha)^{-1}\Big(W_{\alpha,1-\alpha}(u^2t^\alpha)-\frac{1}{\Gamma(1-\alpha)}\Big),\ 0<\alpha<1,\ u\in\mathbb{R},\ t>0,
	\end{equation} 
	where $W_{\nu,\mu}$, $\nu,\mu>0$ is the Wright function (for definition see \cite{Kilbas2006}, Eq. (1.11.1)).
    As a consequence, $\{B^\alpha(t)\}_{t\ge0}$ is a free process with $B^\alpha(0)=0$ and for each $t>0$, the spectral distribution of $B^\alpha(t)$ has a valid probability density with support $\mathbb{R}$, whose associated moment generating function (mgf) is given by (\ref{tcmgf}).
\end{proposition}
\begin{proof}
	The mgf of $Z(t)$ is given by $\mathbb{E}e^{uZ(t)}=I_1(2\sqrt{t}u)/\sqrt{t}u$, $u\in\mathbb{R}$, $t>0$, where $I_\nu$, $\nu>-1$ is the modified Bessel function of first kind (for definition see \cite{Kilbas2006}, Eq. (1.7.16)).
	So,
	\begin{align*}
		\mathbb{E}e^{uZ(L^\alpha(t))}&=\int_{0}^{\infty}\mathbb{E}e^{uZ(x)}\mathrm{Pr}\{L^\alpha(t)\in\mathrm{d}x\}\\
		&=\int_{0}^{\infty}\sum_{r=0}^{\infty}\frac{(\sqrt{x}u)^{2r+1}}{(r+1)!r!\sqrt{x}u}\mathrm{Pr}\{L^\alpha(t)\in\mathrm{d}x\}\\
		&=\sum_{r=0}^{\infty}\frac{u^{2r}}{(r+1)!r!}\int_{0}^{\infty}x^r\mathrm{Pr}\{L^\alpha(t)\in\mathrm{d}x\}\\
		&=\sum_{r=0}^{\infty}\frac{u^{2r}\Gamma(r+1)t^{\alpha r}}{(r+1)!r!\Gamma(\alpha r+1)}\\
		&=(u^2t^\alpha)^{-1}\sum_{r=1}^{\infty}\frac{(u^2t^\alpha)^r}{r!\Gamma(\alpha r+1-\alpha)},
	\end{align*}
	where we have used (\ref{lmomis}) at the penultimate step. This proves the first part. Relation (\ref{tcmgf}) implies that the $l$th  moment of $Z(L^\alpha(t))$ equals (\ref{evnmom}) and the second part follows. 
\end{proof}




	\subsection{Patterned random matrices and link function}\label{rmlf}  Note that the notion of convergence as defined in Section \ref{jointcov} applies to $n \times n$ random matrices with the choice of state  being $n^{-1}\mathbb{E}\text{Trace}$. A related concept of convergence is that of the spectral distribution.  Let $A_n$ be an $n\times n$ symmetric random matrix whose entries are rvs on a probability space $(\Omega,\mathcal{F},\mathbb{P})$, and its eigenvalues are $\lambda_1,\dots,\lambda_n$. The empirical spectral distribution (ESD) of $A_n$ is defined as follows:
\begin{equation}\label{esd}
	F_{A_n}(x)=n^{-1}\sum_{k=1}^{n}\mathbb{I}_{\{\lambda_k\leq x\}},\ x\in\mathbb{R},
\end{equation}
where $\mathbb{I}$ is the indicator function and $\mathbb{R}$ denotes the set of real numbers. The expected empirical spectral distribution (EESD) of $A_n$ is defined as $ \bar{F}_{A_n}(x)\coloneqq\mathbb{E}F_{A_n}(x)$, $x\in\mathbb{R}.$ The limiting spectral distribution (LSD) of $\{A_n\}_{n\ge1}$ is the weak limit of random sequence $\{F_{A_n}\}_{n\ge1}$, whenever it exist. The ESD of $A_n$ is said to be weakly a.s. converge to a non random distribution function $F$ if $\mathbb{P}\{\omega\in\Omega:\ \text{$F_{A_n}(x)\rightarrow F(x)$ as $n\rightarrow\infty$}\}=1$ for all continuity points $x$ of $F$.
    

A class of patterned random matrices that includes many important matrices has been studied in a unified way by \cite{Bose2008}. We introduce their basic ideas now. For $d\in\{1,2\}$, let us consider a sequence $\{L^d_n\}_{n\ge1}$ of functions defined as 
$L^d_n:\{1,2,\dots,n\}^2\rightarrow\mathbb{Z}^d$ 
By using a common notation $L^d\coloneqq L^d_n$, it is called the link function. 
	
Let $\{a_{i},\ i\ge1\}$ and $\{a_{i,j},\ i\ge1,\ j\ge1\}$ be sequence and bi-sequence of rvs. For $n\ge1$, a $n\times n$ random matrix $A_n$ with link function $L^d$ and input sequence $\{a_{i},\ i\ge1\}$ or $\{a_{i,j},\ i\ge1,\ j\ge1\}$ is defined as $A_n\coloneqq((a_{L^d(i,j)}))_{1\leq i,j\leq n}$. The	following are examples of link functions for some well known random matrices:\\
	
	\noindent (i) $L^2(i,j)=(\max\{i,j\},\min\{i,j\})$: Wigner matrix;\\
	\noindent (ii) $L^1(i,j)=|i-j|$: (symmetric) Toepliz matrix;\\
	\noindent (iii) $L^1(i,j)=i+j-2$: (symmetric) Hankel matrix;\\
	\noindent (iv) $L^1(i,j)=i+j-2\ \text{mod}\ n$, $1\leq i,j\leq n$: Reverse Circulant matrix;\\
	\noindent (v) $L^1(i,j)=n/2-|n/2-|i-j||$, $1\leq i,j\leq n$: Symmetric Circulant matrix.\\
    
The following property that all the above link functions satisfy, plays a crucial role:\vskip5pt
\paragraph{\textbf{Property B}} 
	$\Delta L^d\coloneqq \sup_{n}\sup_{r\in\mathbb{Z}^d_+}\sup_{1\leq i\leq n}\#\{j:1\leq j\leq n,\ L^d(i,j)=r\}<\infty$.
\vskip5pt  

From (\ref{esd}), the $l$th moment associated to the ESD of $A_n$ is $\mu^{(l)}_{F_{A_n}}=n^{-1}\mathrm{Trace}(A_n^l)$, $l\ge1$, and the $l$th moment associated to its EESD is given by the trace-moment formula: $\mathbb{E}\mu_{F_{A_n}}^{(l)}=n^{-1}\mathbb{E}\mathrm{Trace}(A_n^l)$.
	Let us now consider a non random sequence $\{\mu^{(l)}\}_{l\ge1}$ such that 
    for all $l\ge1$, $\mathbb{E}\mu^{(l)}_{F_{A_n}}\rightarrow\mu^{(l)}$ as $n\rightarrow\infty$. 
	   If the sequence $\{\mu^{(l)}\}_{l\ge1}$ satisfies Carleman's condition, $\sum_{l=1}^{\infty}(\mu^{(2l)})^{-1/2l}=\infty$,
	then $\{\mu^{(l)}\}_{l\ge1}$ uniquely determines a probability distribution, and in that case the ESD converges in probability to this distribution. The convergence can be strengthened to a.s.~sense if 
    $\sum_{l=1}^{\infty}\mathbb{E}\big(\mu^{(l)}_{F_{A_n}}-\mathbb{E}\mu^{(l)}_{F_{A_n}}\big)^4<\infty$ for all $n\ge1$.
This route has been taken for patterned random matrices. For example, see \cite{Bose2021, Bose2008, Bose2011} and references therein, for several convergence results for patterned random matrices,  both for the ESD and for joint convergence. 

\subsubsection{Matched circuits and words} For $l\ge1$ and $n\ge1$, a function $\pi:\{0,1,\dots,l\}\rightarrow\{1,2,\dots,n\}$ with $\pi(0)=\pi(l)$ is called a circuit of length $l\ge1$. 
A circuit $\pi$ is called matched if all the $L^d$-values $L^d(\pi(j-1),\pi(j))$ are repeated. It is called pair-matched if each value $L^d(\pi(j-1), \pi(j))$ appears exactly twice. We can group these circuits into equivalence classes: circuits $\pi_1$ and $\pi_2$ belong to the same class, denoted as $\pi_1 \sim \pi_2$ if $L^d(\pi_1(j-1), \pi_1(j)) = L^d(\pi_1(i-1), \pi_1(i))$ if and only if $L^d(\pi_2(j-1), \pi_2(j)) = L^d(\pi_2(i-1), \pi_2(i))$ for all $1 \le i, j \le l$. These equivalence classes can be indexed by partitions of $\{1, \dots, l\}$, which are represented by words $w$ of length $l$ in which letters appear in alphabetical order. Let $w(i)$ denote the $i$th letter of $w$ and $\Pi(w)$ be the equivalence class associated with word $w$. Then,
	\begin{equation}\label{eqclsdef}
		\Pi(w)\coloneqq\{\pi: w(i)=w(j)\ \text{if and only if}\ L^d(\pi(j-1), \pi(j))=L^d(\pi(i-1),\pi(i))\}.
	\end{equation}	
A word $w$ is called matched if each letter repeats, and it is called pair-matched if each letter repeats exactly twice. Note that a pair-matched word is always of even length. We denote the set of pair-matched words of length $2m$, $m\ge1$ by $\mathcal{P}_2(2m)$. 
	
	\subsubsection{Joint convergence of independent random matrices}\label{jtcov} 
    We now outline the idea of joint convergence of independent symmetric patterned matrices using the trace-moment formula. Let $A_{r,n}$, $r=1,2,\dots,p$ be $n\times n$ independent random matrices with common link function $L^d$.	The $(i,j)$th entry of $A_{r,n}$ will be denoted by $A_{r,n}(L(i,j))$. Note that a typical monomial $g$ in terms of $\{A_{r,n}\}_{1\leq r\leq p}$ has the following form: $g(\{A_{r,n}\}_{1\leq r\leq p})=A_{r_1,n}\dots A_{r_l,n}$, where $r_j\in\{1,2,\dots,p\}$ for each $j=1,2,\dots,l$. The collection $\{A_{r,n}\}_{1\leq r\leq p}$ in $(\mathcal{M}_n(\mathbb{C}), \tau_n=n^{-1}\mathbb{E}\mathrm{Trace})$ converges jointly if for every fixed monomial $g$, the following converges:
	\begin{equation*}
	\tau_n(g)
	=n^{-1}\sum_{\pi:\,\text{$\pi$ is a circuit of length $m$}}\mathbb{E}A_{r_1,n}(L^d(\pi(0),\pi(1)))\dots A_{r_ln}(L^d(\pi(l-1),\pi(l))).
	\end{equation*}
	It turns out that monomials $g$ with odd number of factors always yield the limit $0$ and hence without loss we assume that the order of $g$ is $2m$ for some integer $m$. A circuit $\pi$ of length $l$ is called index-matched if the value $(r_j,L^d(\pi(j-1),\pi(j)))$ repeats for each $j=1,2,\dots,l$. Note that circuits are index-matched with reference to the monomial $g$. Let $I=\{\pi:\ \pi\ \text{is a index-matched circuit}\}$. Consider a equivalence class on $I$ where for any two circuits $\pi_1$ and $\pi_2$, we have $\pi_1\sim\pi_2$ if for any $1\leq i,j\leq l$ such that $r_i=r_j$, $A_{r_i,n}(L^d(\pi_1(i-1),\pi_1(i)))=A_{r_j,n}(L^d(\pi_1(j-1),\pi_1(j)))$ if and only if $A_{r_i,n}(L^d(\pi_2(i-1),\pi_2(i)))=A_{r_j,n}(L^d(\pi_2(j-1),\pi_2(j)))$. Any such equivalence class can be indexed by a partition of $\{1,2,\dots,l\}$, which further induces an indexed word $w$ of length $l$, where the indexed letters appear in alphabetical order. An indexed word is called pair-matched if every indexed letter in it appears exactly twice. Only such letters remain relevant in the limit. Let $w_{r_j}(j)$ denote the $j$th letter of indexed matched word $w$ and $\Pi_g(w)$ denote the class of circuits corresponding to $w$ and monomial $g$. Then,
	\begin{equation*}
		\Pi_g(w)=\{\pi:\ w_{r_i}(i)=w_{r_j}(j)\ \text{iff}\ A_{r_i,n}(L^d(\pi(i-1),\pi(i)))=A_{r_j,n}(L^d(\pi(j-1),\pi(j)))\}.
	\end{equation*}
The set of indexed pair-matched words ($g$ is necessarily of length $2m$) is denoted by
	\begin{equation*}
		\mathcal{P}_2(2m,g)=\{\text{all indexed pair-matched word of length $2m$, $m\ge1$}\}.
	\end{equation*}
If we drop the indices of the indexed word $w$, then it becomes a matched word. Thus, for each monomial $g$ (of order $2m$) there exist a bijection $\chi_g:\mathcal{P}_2(2m,g)\rightarrow\mathcal{P}_2(2m)$.

Using the results for a single matrix, this would show the joint convergence of the iid copies of the pattern matrices described above. For further details, see \cite{Bryc2006, Bose2021, Bose2008}, and references therein.

\section{Patterned random matrices with CTRW entries} \label{sec3} Motivated by all the discussions so far, let $\{X_{i,j,k}, i\ge1,\,j\ge1,\, k\ge1\}$ and $\{X_{i,k}, i\ge0, \, k\ge1\}$ be sequences of real valued rvs. Let us consider the collection of CTRWs $\{\{S_{i,j}^{(n)}(t)\}_{t\ge0}\}_{1\leq i,j\leq n}$ and $\{\{S_{i}^{(n)}(t)\}_{t\ge0}\}_{1\leq i,j\leq n}$, $n\in\mathbb{N}$ defined as follows:
\begin{eqnarray}
	S_{i,j}^{(n)}(t)\coloneqq&\sum_{k=1}^{[nt]}X_{i,j,k},\ 1\leq i,j\leq n,\ n\in\mathbb{N}, \label{ctrw1}\\
	S_{i}^{(n)}(t)\coloneqq&\sum_{k=1}^{[nt]}X_{i,k},\ 0\leq i\leq n,\ n\in\mathbb{N}.\label{ctrw11}
  \end{eqnarray}  
Let us consider the sequence $\{\{A_n(t)\}_{t\ge0}\}_{n\ge1}$ of matrix valued random processes with link function $L^d$ defined as follows: 
\begin{equation}\label{matrixdef}
	A_n(t)=n^{-1/2}((S_{L^d(i,j)}^{(n)}(t)))_{1\leq i,j\leq n},\ n\ge1,\ t\ge0.
\end{equation} 
$A_n(0)$ is taken to be the $n\times n$ matrix of zeroes. Depending on whether $d=1$ or $d=2$, the entries of $A_n(t)$ are from the sequences in (\ref{ctrw1}) or (\ref{ctrw11}) respectively. Consider the scaled matrix $n^{-1/2}A_n(t)$, with eigenvalues are $\lambda_1(t),\dots,\lambda_n(t)$. Its ESD  is given by $F_{n^{-1/2}A_n(t)}(x)\coloneqq n^{-1}\sum_{r=1}^{n}\mathbb{I}(\lambda_r(t)\leq x),\ x\in\mathbb{R}$, $t >0$. \vskip5pt

The following assumption will be needed:\vskip5pt

\noindent \textbf{Assumption I}\ $\{ X_{i,j,k}, X_{i,k}, \ i\ge0,\ j\ge0,\ k\ge1\}$ are independent rvs with zero mean and unit variance such that
$\sup_{i,j,k}\mathbb{E}[|X_{i,j,k}|^l+|X_{i,k}|^l]<K_l<\infty$ for all $l\ge3$.\vskip5pt

\subsection{Marginal algebraic convergence}\label{marginal}

Under Assumption I, the random variables $n^{-1/2}\sum_{k=1}^n X_{i,j,k}$ and $n^{-1/2}\sum_{k=1}^n X_{i,k}$ are independent  with mean $0$ and variance $1$.


Note that for a fixed $t$, the matrix $A_n(t)$ is a sum of $[nt]$-many independent symmetric random matrices, and as $n\rightarrow\infty$, the number of summands and their dimensions are increasing. Moreover, under Assumption I, the ESD of each summand converges under appropriate conditions (see Theorem 7.5.1 of \cite{Bose2021}). In the following result, we show that under Assumption I, the LSD of $n^{-1/2}A_n(t)$ coincides with that of its summands. Its proof is a particular case of the proof of Theorem \ref{thm3}. Hence, we omit it.

\begin{theorem}\label{thm1}	Suppose $\{X_{i,j,k}, \ X_{i,k},\ 0\leq i\leq j\leq n,\ k\ge1\}$, $n\in\mathbb{N}$  satisfies Assumption I.  Let $A_n(t)$ be the random matrix as defined in (\ref{matrixdef}) with symmetric link function $L^d$ that satisfies Property B. Assume that the following limits exist:
	\begin{equation}\label{classlim}
		\lim_{n\rightarrow\infty}\frac{\#\Pi(w)}{n^{1+m}}\ \text{for all $w\in\mathcal{P}_2(2m)$, $m\ge1$},
	\end{equation}
	where $\#\Pi(w)$ denotes the cardinality of the equivalence class $\Pi(w)$ as defined in (\ref{eqclsdef}).	Then, 
	\begin{equation}\label{thm1re1}
		\mu^{(l)}(t)\coloneqq\lim_{n\rightarrow\infty}\frac{\mathbb{E}\mathrm{Trace}(A_n(t))^l}{n^{1+l/2}}=\begin{cases}
			0,\ l=2m-1,\ m\ge1,\vspace{0.2cm}\\
			\displaystyle t^m\sum_{w\in\mathcal{P}_2(2m)}\lim_{n\rightarrow\infty}\frac{\#\Pi(w)}{n^{1+m}},\ l=2m,\ m\ge1.
		\end{cases}
	\end{equation}
 Thus,  the EESD of $n^{-1/2}A_n(t)$ weakly converges to the unique probability law whose $l$th moment equals $\mu^{(l)}(t)$ for each $l\ge1$. As a consequence, the self-adjoint element $n^{-1/2}A_n(t)$ of the NCP $(\mathcal{M}_n(\mathbb{R}),\tau_n=n^{-1}\mathbb{E}\mathrm{Trace})$ converges to an element $A(t)$ in $*$-algebra $\mathcal{A}(t)$ generated by $A(t)$ with state $\tau$ defined as $\tau(A^l(t))=\mu^{(l)}(t)$.
\end{theorem}

\begin{theorem}\label{thm2as}
	For $t\ge0$, let $A_n(t)$ be as  in Theorem \ref{thm1}. Then, 
	\begin{equation*}
		\mathbb{E}\big(n^{-1}\mathrm{Trace}(n^{-1/2}A_n(t))^l-\mathbb{E}n^{-1}\mathrm{Trace}(n^{-1/2}A_n(t))^l\big)^4=t^{2l}O(n^{-2}),\ l\ge1.
	\end{equation*}
	Thus, the ESD of $n^{-1/2}A_n(t)$ converges weakly a.s. to a unique probability law with moments as in (\ref{thm1re1}).
\end{theorem}

\begin{remark}\label{rem31} Fix a symmetric link function $L^d$ that satisfying Property B.
Let $A_{1,n}, A_{2,n}, \dots$ be independent $n\times n$ random matrices with link function $L^d$ and variables $\{X_{i,j,1}\}, \{X_{i,j,2}\},\dots$ or  
$\{X_{i,1}\}, \{X_{i,2}\},\dots$, depending on whether $d=1$ or $d=2$.  
Suppose the variables satisfy Assumption I. If the limits (\ref{classlim}) exist, then for each $j\ge1$, the EESD of $B_{j,n}=n^{-1/2}A_{j,n}$ weakly converges to a probability law whose $l$th moment equals $\mu^{(l)}(1)$ for every $l\ge1$, defined in (\ref{thm1re1}). Moreover, this convergence holds true for their ESD  a.s. (see \cite{Bose2021}).
      
From Theorem \ref{thm1} and Theorem \ref{thm2as} for $t=1$, it follows that the ESD of $n^{-1/2}(B_{1,n}+\dots+B_{n,n})$ also converges weakly a.s. to the probability law whose $l$th moment is given by $\mu^{(l)}(1)$ for all $l\ge1$. Thus, each $B_{j,n}$ and $n^{-1}(A_{1,n}+\dots+ A_{n,n})$ have same LSD.
\end{remark}

We now collect the LSD of some specific pattern random matrices with CTRW entries in the following corollary, In view of Theorem \ref{thm1}, its proof follows from the results of \cite{Bose2008}.   

\begin{corollary}\label{thm2}
Suppose $\{X_{i,j,k},\ X_{i,k},\ 1\leq i\leq j\leq n,\ k\ge1\}$, $n\in\mathbb{N}$ 
satisfies Assumption I. Let $\{\{A_n(t)\}_{t\ge0}\}_{n\ge1}$ be 
as defined in (\ref{matrixdef}) with link function $L^d$. Then, for each $t$,\vskip3pt

\noindent (i) for $L^2$ defined by $L^2(i,j)=(\max\{i,j\},\min\{i,j\})$, the ESD of $n^{-1/2}A_n(t)$ converges a.s.~to the semi-circular law with zero mean and variance $t$;\vskip3pt

\noindent (ii) If $L^1(i,j)=|i-j|$ then the ESD of $n^{-1/2}A_n(t)$ converges a.s.~to a probability law whose corresponding $(2m)$th moment is a sum of volumes of polyhedra in some $(k+1)$-dimensional hypercube;\vskip3pt

\noindent (iii) If $L^1(i,j)=i+j-2$ then the ESD of $n^{-1/2}A_n(t)$ converges a.s.~to a probability law, symmetric about zero, whose corresponding $2m$th moment is a sum of volumes of polyhedra in some $(k+1)$-dimensional hypercube;\vskip3pt

\noindent (iv) If $L^1(i,j)=i+j-2\ \text{mod}\ n$, $1\leq i,j\leq n$ then the ESD of $n^{-1/2}A_n(t)$ converges a.s.~to a symmetric Rayleigh law whose $(2m)$th moment is given by $t^mm!$,\ $m\ge1$;\vskip3pt

\noindent (v) If $L^1(i,j)=n/2-|n/2-|i-j||$, $1\leq i,j\leq n$ then the ESD of $n^{-1/2}A_n(t)$ converges a.s.~to the Gaussian law with mean zero and variance $t$.\vskip5pt

\end{corollary}

\subsection{Joint algebraic convergence}\label{joint} Now we study the finite dimensional algebraic convergence of the matrix valued process $\{n^{-1/2}A_n(t)\}_{t\ge0}$. It is enough to consider the joint convergence of its increments. For $0\leq s<t<\infty$, an increment of $\{A_n(t)\}_{t\ge0}$ is given by
\begin{equation*}
	A_n(t)-A_n(s)\coloneqq n^{-1/2}\big((S_{L^d(i,j)}^{(n)}(t)-S_{L^d(i,j)}^{(n)}(s))\big)_{1\leq i,j\leq n},\ n\ge1.
\end{equation*}
As $\{S_{i,j}^{(n)}(t)\}_{t\ge0}$'s have independent increments, $\{A_n(t)\}_{t\ge0}$ also has independent increments.

For $p\in\mathbb{N}$, let $\mathbb{T}_p=\{t_0,t_1,\dots,t_p\}\subset[0,\infty)$ with $0\leq t_0<t_1<\dots<t_p$. Let $\{A(t)\}_{t\ge0}$ be a process associated to an NCP $(\mathcal{A},\tau)$, and let
\begin{equation}\label{incset}
    \mathbb{T}_p(A(t))\coloneqq\{A_{1},\dots,A_{p}\}
\end{equation}
be the collection of increments of $\{A(t)\}_{t\ge0}$ over $\mathbb{T}_p$, where $A_{r}=A(t_r)-A(t_{r-1})$ for each $r=1,2,\dots,p$.

\begin{theorem}\label{thm3}
Let $\{A_n(t)\}_{t\ge0}$ be as in Theorem \ref{thm1}, and let $\mathbb{T}_p(n^{-1/2}A_n(t))$ be the collection of increments of $\{n^{-1/2}A_n(t)\}_{t\ge0}$ over $\mathbb{T}_p$ for $p\in\mathbb{N}$, as defined in (\ref{incset}). Let $g$ be a monomial given by $g(\mathbb{T}_p(n^{-1/2}A_n(t)))\coloneqq A_{r_1,n}\dots A_{r_l,n}$, where $A_{r_j,n}\in\mathbb{T}_p(n^{-1/2}A_n(t))$ for each $j=1,2,\dots,l\ge1$. Also, let us assume that the following limits exist:
	\begin{equation}\label{limcond2}
		\lim_{n\rightarrow\infty}\frac{\#\Pi(\chi_g(w))}{n^{1+m}}\ \text{for all $w\in\mathcal{P}_2(2m,g)$},\ m\ge1,
	\end{equation}
	where $\mathcal{P}_2(l,g)$ and $\chi_g$ are as defined in Section \ref{jtcov}.\\
 \noindent (i) Then, 
	\begin{equation}\label{thm3re1}
		\lim_{n\rightarrow\infty}\frac{\mathbb{E}\mathrm{Trace}(A_{r_1,n}\dots A_{r_l,n})}{n}=\begin{cases}
			0,\ l=2m-1,\\
			\displaystyle\prod_{\substack{j=1
            }}^{m}(t_{r_j}-t_{r_j-1})\sum_{w\in\mathcal{P}_2(2m,g)}\lim_{n\rightarrow\infty}\frac{\#\Pi(\chi_g(w))}{n^{1+m}},\ l=2m
		\end{cases}
	\end{equation}
    for all $m\ge1$.
    Consequently,
we can fix indeterminates $A(t)$ for every $t \geq 0$ and define an algebra     $\tilde{\mathcal{A}}\coloneqq\mathrm{alg}\{\cup_{p\in\mathbb{N}}\cup_{\mathbb{T}_p}\mathcal{A}(\mathbb{T}_p)\}$, generated by the union $\cup_{p\in\mathbb{N}}\cup_{\mathbb{T}_p}\mathcal{A}(\mathbb{T}_p)$, where $\mathcal{A}(\mathbb{T}_p)$ is a polynomial algebra over $\mathbb{C}$, generated by the set $\{A(t): t\in\mathbb{T}_p\}$, and define the state $\tilde{\tau}$ on it as follows:	
    \begin{equation}\label{tildetau}
		\tilde{\tau}(g(\{A(t): t\in\mathbb{T}_l\}))\coloneqq\lim_{n\rightarrow\infty}\tau_n(g(\{n^{-1/2}A_n(t): t\in\mathbb{T}_l\})),
        \end{equation}
        for all $l\ge1$ and all polynomials $g$. Then the process $\{n^{-1/2}A_n(t)\}_{t\ge0}$ associated to the NCP     $(\mathcal{M}_n(\mathbb{C}), n^{-1}\mathbb{E}\mathrm{Trace})$ converges algebraically to $\{A(t)\}_{t\ge0}$ associated to  $(\tilde{\mathcal{A}},\tilde{\tau})$. \\   
    
\noindent (ii)   If  $L^2(i,j)=(\max\{i,j\},\min\{i,j\})$ then (\ref{limcond2}) holds and the collection $\{A(t)\}_{t\ge0}$ in $\tilde{\mathcal{A}}$ is a FBM. Similarly, if $L^1(i,j)=n/2-|n/2-|i-j||$, $1\leq i,j\leq n$ then the joint moments of $\{A(t)\}_{t\ge0}$ agree with those of a BM. 
\end{theorem}

\begin{proof}
 (i) Note that the elements of $\mathbb{T}_p(n^{-1/2}A_n(t))$ are mutually independent. 
For a circuit $\pi$ of length $l$, let 
\begin{equation*}
	S_\pi^{(n)}(\mathbb{T}_p)=\mathbb{E}(S_{L^d(\pi(0),\pi(1))}^{(n)}(t_{r_1})-S_{L^d(\pi(0),\pi(1))}^{(n)}(t_{r_1-1}))\dots (S_{L^d(\pi(l-1),\pi(l))}^{(n)}(t_{r_l})-S_{L^d(\pi(l-1),\pi(l))}^{(n)}(t_{r_l-1})).
\end{equation*}
Then, 
\begin{equation}\label{thm3pf1}
n^{-1}\mathbb{E}\mathrm{Trace}(A_{r_1,n}\dots A_{r_l,n})=\frac{1}{n^{1+l}}\sum_{\pi:\,\pi\,\text{is circuit of length $l$}}S_{\pi}^{(n)}(\mathbb{T}_p).
\end{equation}
Note that if any index $r_i$, $i=1,\dots,l$ appears only once or circuit $\pi$ is non matched then by independent increments property and zero mean of $\{S_{i,j}^{(n)}(t)\}_{t\ge0}$, we get $S_{\pi}^{(n)}(\mathbb{T}_p)=0$. Thus, $S_{\pi}^{(n)}(\mathbb{T}_p)=0$ for all non indexed matched circuits $\pi$. Let us write $M_{L^d(i,j)}^{(n)}(\mathbb{T}_p)=\sum_{l'=1}^{l}\big|S_{L^d(i,j)}^{(n)}(t_{r_{l'}})-S_{L^d(i,j)}^{(n)}(t_{r_{l'}-1})\big|$, $1\leq i,j\leq n$.
Then, 
\begin{equation*}
	|S_\pi^{(n)}(\mathbb{T}_p)|\leq\mathbb{E}M_{L^d(\pi(0),\pi(1))}^{(n)}(\mathbb{T}_p)\dots M_{L^d(\pi(l-1),\pi(l))}^{(n)}(\mathbb{T}_p).
\end{equation*}
For $h\ge1$ and $s<t$, we have $\mathbb{E}|S_{L^d(i,j)}^{(n)}(t)-S_{L^d(i,j)}^{(n)}(s)|^h\leq K_h([nt]-[ns])^{[h/2]}$, which implies $\mathbb{E}(M_{L^d(i,j)}^{(n)}(\mathbb{T}_p))^h\leq \prod_{l'=1}^{l}K_{h_{l'}}([nt_{r_{l'}}]-[nt_{r_{l'}-1}])^{[h_{l'}/2]}$, where $h_1+\dots+h_l=h$. Thus, for any non pair matched circuit $\pi$ of length $l$, following the arguments of the proof for Theorem \ref{thm1}, we get $|S_\pi^{(n)}(\mathbb{T}_p)|\leq K_l\prod_{l''=1}^{l}([nt_{r_{l'}}]-[nt_{r_{l'}-1}])^{[h_{l''}'/2]}$, where $h_1'+\dots+h_l'=l$. Now, on using $\#\Pi(w)\leq K_ln^{(l+1)/2}$ for any non pair-matched word $w$ of length $l$ (see \cite{Bose2021}, p. 120), we conclude that the contribution of a non indexed pair-matched circuit in (\ref{thm3pf1}) is zero as $n\rightarrow\infty$. Hence,
\begin{equation*}
	\lim_{n\rightarrow\infty}n^{-1}\mathbb{E}\mathrm{Trace}(A_{r_1,n}\dots A_{r_{2m},n})=\sum_{w\in\mathcal{P}_2(2m,g)}\lim_{n\rightarrow\infty}\frac{\#\Pi(\chi_g(w))}{n^{1+2m}}S_{\pi_{g,w}}^{(n)},\ m\ge1,
\end{equation*}
where $\pi_{g,w}\in\Pi(\chi_g(w))$ is a fixed indexed pair-matched circuit of length $2m$. Moreover, 
\begin{align*}
	S_{\pi_{g,w}}^{(n)}(\mathbb{T}_p)&=\mathbb{E}(S_{j_1,j_2}^{(n)}(t_{r_1})-S_{j_1,j_2}^{(n)}(t_{r_1-1}))^2\dots\mathbb{E} (S_{j_m,j_1}^{(n)}(t_{r_m})-S_{j_m,j_1}^{(n)}(t_{r_m-1}))^2\\
	&= ([nt_{r_1}]-[nt_{r_1-1}])\dots([nt_{r_m}]-[nt_{r_m-1}]),\ w\in\mathcal{P}_2(2m,g),\ m\ge1,
\end{align*}
So,
\begin{equation*}
	\lim_{n\rightarrow\infty}\frac{\#\Pi(\chi_g(w))}{n^{1+2m}}S_{\pi_{g,w}}^{(n)}(\mathbb{T}_p)=\lim_{n\rightarrow\infty}\frac{\#\Pi(\chi_g(w))}{n^{1+m}}(t_{r_1}-t_{r_1-1})\dots(t_{r_m}-t_{r_m-1}),\ w\in\mathcal{P}_2(2m,g).
\end{equation*}
This completes the proof of (i).\\

\noindent (ii) Clearly $A(0)=0$. For $0\leq s<t<\infty$ and $A(s)$ and $A(t)$ in $\tilde{\mathcal{A}}$, from (\ref{thm3re1}), we get
	\begin{equation*}
		\lim_{n\rightarrow\infty}\frac{\mathbb{E}\mathrm{Trace}(A_n(t)-A_n(s))^l}{n^{1+l/2}}=\begin{cases}
			0,\ l=2m-1,\ m\ge1,\\
			\displaystyle(t-s)^m\sum_{w\in\mathcal{P}_2(2m)}\lim_{n\rightarrow\infty}\frac{\#\Pi(w)}{n^{1+m}},\ l=2m,\ m\ge1.
		\end{cases}
	\end{equation*}
    Thus, from (\ref{thm1re1}), we get $\tilde{\tau}((A(t)-A(s))^l)=\tilde{\tau}((A(t-s))^l)$ for all $l\ge1$. 
 
  If $L^2(i,j)=(\max\{i,j\},\min\{i,j\})$ then from Corollary \ref{thm2}(i), it follows that for all $s<t$, $A(t)-A(s)$ is a semi-circle element with mean zero and variance $t-s$. Moreover, in view of part (i), the variables in $\mathbb{T}_p(A(t))$, defined in (\ref{incset}), are freely independent using Theorem 9.2.1 of \cite{Bose2021}. 
  
 For the Symmetric circulant, similar arguments work. Convergence of the moments of any monomial to the corresponding moments of a BM follows from Corollary \ref{thm2} (v) and Theorem 9.5.1 of \cite{Bose2021}. This completes the proof.
\end{proof}

\begin{remark}\label{remarkpcopies1}
(i) In \cite{Banerjee2013}, a general condition is obtained on the link function of random matrices that are not necessarily Wigner matrices, under which the associated LSD has the semi-circular density.
Let $A_n(t)$ be the random matrix with link function $L^d$, as defined in Theorem \ref{thm1}. Define 
	\begin{equation*}
		\beta_n=\max_{r}\#\{(i,j):L^d(i,j)=r,\ 1\leq i,j\leq n\},\ n\ge1.
	\end{equation*}
	Let us consider the following matched set between the columns $i$ and $j$ at stage $n$:
	\begin{equation*}
		M_{i,j}^{(n)}=\{k: 1\leq k\leq n, L^d(k,i)=L^d(k,j)\},
	\end{equation*}
	If $M^*=\sup_{n\ge1}\sup_{1\leq i,j\leq n}\# M_{i,j}^{(n)}<\infty$,  and $\beta_n=o(n)$,  then in view of Theorem 2 of \cite{Banerjee2013}, the LSD of $A_n(t)$ has semi-circle density with variance $t>0$. Indeed, Part (ii) of Theorem \ref{thm3} for the corresponding process $\{A_n(t)\}_{t\geq 0}$ also holds in this case. We omit the proof.  \vskip5pt

\noindent (ii) Part (ii) of Theorem  \ref{thm3} can be extended in another direction. Consider $p$ independent copies of $\{A_n(t)\}_{t\geq 0}$ for the Wigner link function, say $L^2$ and the Symmetric circulant link function, say $L^1$. Then for $L^2$, these $p$ copies converge in $*$-distribution and are asymtotically free.
For $L^1$, these $p$ copies converge in $*$-distribution and are asymptotically independent.
\end{remark}

\subsection{Symmetric circulant and weak convergence to a BM}\label{symcircsbm} 
Now we consider $A_n(t)$'s with the link function $L^1(i,j)=n/2-|n/2-|i-j||$, $1\leq i,j\leq n$, in more details. 
Then $F_{n^{-1}A_n(t)}, t\ge0$, the ESD of $A_n(t), t\ge0$ is a random process. 
Suppose Assumption I holds. Then we already know that for each $t>0$, the ESD $F_{n^{-1}A_n(t)}$ converges weakly a.s.~to the Gaussian law with mean zero and variance $t$. 

The eigenvalues $\lambda_{1,n}(t),\dots,\lambda_{n,n}(t)$ of the Symmetric Circulant matrix $n^{-1/2}A_n(t)$ are given as follows (see \cite{Bose2018}, Section 4.2):
\begin{equation*}
    \lambda_{r,n}(t)=\frac{S_0^{(n)}}{n}+\frac{2}{n}\sum_{j=1}^{[n/2]}S_j^{(n)}(t)\cos{\bigg(\frac{2\pi rj}{n}\bigg)},\ t>0,\ r=1,\dots,n,
\end{equation*}
where $S_j^{(n)}(t)=\sum_{k=1}^{[nt]}X_{j,k}$ is as defined in (\ref{ctrw11}).

Let $U_n$ be a discrete uniform rv on $\{1,2,\dots,n\}$ that is independent of $\{\lambda_{r,n}(t)\}_{t\ge0}$ for all $n\ge1$ and $1\leq r\leq n$. Let us now define a sequence of random processes $\{Y_n(t)\}_{t\ge0}$, $n\ge1$ as follows:
\begin{equation}\label{newp}
    Y_n(t)\coloneqq \lambda_{U_n,n}(t)=n^{-1}\sum_{r=1}^n \lambda_{r,n}\mathbb{I}(U_n=r),\ \text{for all}\ t\ge0.
\end{equation}
 In the next result, we establish the functional convergence of $\{Y_n(t)\}_{t\ge0}$, as $n\rightarrow\infty$.
 \begin{theorem}\label{thmwc1}
     Let $\{W(t)\}_{t\geq 0}$ be a BM, and let $\{Y_n(t)\}_{t\geq 0}$ be the random process, as defined in (\ref{newp}). Then, 
     $\{Y_n(t)\}_{t\geq 0}$ converges weakly to $\{W(t)\}_{t \geq 0}$, as $n\rightarrow\infty$.
 \end{theorem}
 \begin{proof}
 First, we show that the process $\{Y_n(t)\}_{t\ge0}$ convergences to  $\{W(t)\}_{t\ge0}$ in finite dimensional sense, that is, $
         (Y_n(t):t\in\mathbb{T}_p)\overset{d}{\rightarrow}(W(t):t\in\mathbb{T}_p)$
     for all $p\in\mathbb{N}$, where $\overset{d}{\rightarrow}$ denotes the convergence in distribution.
 For $m_1,\dots,m_p\in\mathbb{N}$, we have
 \begin{align*}
     \mathbb{E}\big[Y_n(t_1)^{m_1}\dots Y_n(t_p)^{m_p}\big]&=\frac{1}{n}\sum_{r=1}^{n}\mathbb{E}\lambda_{r,n}(t_1)^{m_1}\dots\lambda_{r,n}(t_p)^{m_p}\\
     &=\frac{1}{n}\mathbb{E}\mathrm{Trace}((n^{-1/}A_{n}(t_1))^{m_1}\dots (n^{-1/2}A_n(t_p))^{m_p}),
 \end{align*}
 where we have used the fact that symmetric circulant matrices commute and are simultaneously diagonalizable using the discrete Fourier matrix. Hence, the finite dimensional convergence follows from Theorem \ref{thm3} (note that the limit moments are Gaussian moments and hence they determine the finite dimensional distributions uniquely).

 Now, we need to show that the sequence $\{Y_n(t)\}_{t\ge0}$, $n\ge1$ is tight. For $0\leq s\leq t<\infty$, we have
    \begin{equation*}
        |\lambda_{r,n}(t)-\lambda_{r,n}(s)|=\frac{2}{n}\bigg|\sum_{j=1}^{[n/2]}\cos\bigg(\frac{2\pi rj}{n}\bigg)(S_j^{(n)}(t)-S_j^{(n)}(s))\bigg|.
    \end{equation*}
    Then, for $0\leq t_1\leq t_2\leq t_3<\infty$, using (\ref{newp}) and the independent increments property of $\{S_j^{(n)}(t)\}_{t\ge0}$, we get
    \begin{align*}
        \mathbb{E}|Y_{n}(t_3)-Y_{n}(t_2)|^2&|Y_{n}(t_2)-Y_{n}(t_1)|^2\\
        &=\frac{1}{n}\sum_{r=1}^{n}\mathbb{E}|\lambda_{r,n}(t_3)-\lambda_{r,n}(t_2)|^2|\lambda_{r,n}(t_2)-\lambda_{r,n}(t_1)|^2\\
        &\leq \frac{4}{n^5}\sum_{r=1}^{n}\sum_{1\leq j_1,j_2\leq [n/2]}\mathbb{E}\big|(S_{j_1}^{(n)}(t_3)-S_{j_1}^{(n)}(t_2))\big|\big|(S_{j_2}^{(n)}(t_3)-S_{j_2}^{(n)}(t_2))\big|\\
        &\ \ \cdot\sum_{1\leq j_3,j_4\leq [n/2]}\mathbb{E}\big|(S_{j_3}^{(n)}(t_2)-S_{j_3}^{(n)}(t_1))\big|\big|(S_{j_4}^{(n)}(t_2)-S_{j_4}^{(n)}(t_1))\big|\\
        &\leq \frac{4}{n^4}\sum_{j_1=1}^{[n/2]}\mathbb{E}\big|(S_{j_1}^{(n)}(t_3)-S_{j_1}^{(n)}(t_2))\big|^2
        \sum_{j_3=1}^{[n/2]}\mathbb{E}\big|(S_{j_3}^{(n)}(t_2)-S_{j_3}^{(n)}(t_1))\big|^2\\
        &\leq \frac{4}{n^4}[n/2]([nt_3]-[nt_2])[n/2]([nt_2]-[nt_1])\leq K(t_3-t_1)^2,
    \end{align*}
    where $K$ is a constant, independent of the choice of $n$. Thus, the tightness of $\{Y_n(t)\}_{t\ge0}$ follows from Theorem 13.5 and Eq. (13.14) of \cite{Billingsley1999}. This completes the proof.
\end{proof}
\begin{remark}\label{remarkpcopies2} As in Remark \ref{remarkpcopies1} (ii), Theorem \ref{thmwc1} can be extended to
$p$ independent copies of symmetric circulant matrices that satisfy Assumption I. Define the joint empirical process in the natural way by using the same $U_n$ for all the $p$ matrices. Then this process converges weakly a.s.~to $p$ independent standard Brownnian motions. Later we shall study the usual circulant for which we prove similar results. Since those arguments are more general than what is required for the present case, we skip the details here. 
\end{remark}
\section{Patterned matrices with randomly stopped walks}\label{rmtstoppedrw}

Let $X_1, X_2,\dots$ be iid rvs that are independent of the fPP $\{N^{\alpha}(t)\}_{t \geq 0}$. Consider the scaled continuous time random walk $\{c^{-\alpha/2}\sum_{k=1}^{N^\alpha(ct)}X_k\}_{t\ge0}$, $c>0$. It is known that if $X_k$'s belong to the domain of attraction of standard Gaussian distribution then $\{c^{-\alpha/2}\sum_{k=1}^{N^\alpha(ct)}X_k\}_{t\ge0}$ weakly converges to a time-changed process $\{W(L^\alpha(t))\}_{t\ge0}$ in an appropriate Skorokhod topology, as $c\rightarrow\infty$ (see \cite{Meerschaert2012}) where $\{W(t)\}_{t\ge0}$ is a BM that is independent of an inverse $\alpha$-stable subordinator $\{L^\alpha(t)\}_{t\ge0}$ as in Section \ref{fpp}. 

In this section, we study the convergence of random matrices whose entries are randomly stopped independent random walks. 
Let $\{X_{i,j,k}, \ i\ge1,\ j\ge1,\ k\ge1\}$ and $\{X_{i,k} \ i\ge0,\ k\ge1\}$ be collections of independent rvs that are independent of $\{N^\alpha(t)\}_{t\ge0}$. We consider 
continuous time random walks defined as follows: 
\begin{eqnarray}
	\mathcal{S}_{i,j,n}^\alpha(t)\coloneqq &\sum_{k=1}^{N^\alpha(nt)}X_{i,j,k},\ t\ge0,\ i\ge1,\ j\ge1,\ n\ge1\label{ctrw2}, \\
    \mathcal{S}_{i,n}^\alpha(t)\coloneqq &\sum_{k=1}^{N^\alpha(nt)}X_{i,k},\ t\ge0,\ i\ge0,\  n\ge1 \label{ctrw22}.
\end{eqnarray}
Note that (\ref{ctrw2}) and (\ref{ctrw22}) are fractional compound Poisson processes, and are compound Poisson processes for $\alpha=1$. Consider the patterned random matrices $\{A_{n}^\alpha(t)\}_{t\ge0}$ as
\begin{equation}\label{ctrm2}
	A_n^\alpha(t)\coloneqq n^{-\alpha/2}((\mathcal{S}^\alpha_{L^d(i,j), n}(t)))_{1\leq i,j\leq n},\ 0<\alpha\leq 1,\ n\ge1.
\end{equation} 
We also assume that $\{X_{i,j,k}, X_{i,k},\ 0\leq i\leq j,\ k\ge1\}$ satisfies Assumption I and the link function $L^d$ 
is symmetric, and satisfies Property B. 

In commutative probability, a Markov process, time-changed by an inverse stable subordinator, inherently produces a process whose inter-event times have heavy tailed distributions. 
For example, the fPP, see Section \ref{fpp}. Expecting something similar, Figures \ref{fig1} and \ref{fig2} illustrate the ESDs of $n^{-1/2} A_n^{\alpha}(t)$ for different values of $\alpha \in (0,1]$, corresponding to Wigner and Symmetric Circulant type link functions, respectively. It is observed that, 
for smaller values of $\alpha$, the eigenvalue spectrum is more widely spread, whereas increasing $\alpha$ leads to a contraction in the range of the $x$-axis. That is, the tails of the ESDs diminish as $\alpha$ increases.

For $\alpha=1$, entries of $A^{\alpha}_n(t)|_{\alpha=1}$ are $n^{-1/2}\mathcal{S}_{L^d(i,j),n}(t)=n^{-1/2}\sum_{k=1}^{N(nt)}X_{L^d(i,j),k}$, where $\{N(t)\}_{t\ge0}$ is a Poisson process. For each $1\leq i,j\leq n$, the increments in $\mathbb{T}_p(\mathcal{S}_{L^d(i,j),n}(t))$ involve $X_{i,j,k}$'s or $X_{i,k}$'s indexed by Poisson points in disjoint intervals. Hence these increments 
as well as those in 
in $\mathbb{T}_p(A_n^1(t))$ are independent. The following result on  $\mathbb{T}_p(n^{-1/2}A_n^1(t))$ is similar to Theorem \ref{thm3}:
\begin{theorem}\label{thm43}
	Let $A^1_n(t)=A^\alpha_n(t)|_{\alpha=1}$ be as in (\ref{ctrm2}). For $1\leq l<\infty$, let $g(\mathbb{T}_p(n^{-1/2}A_n^1(t)))\coloneqq A_{r_1,n}^1\dots A_{r_l,n}^1$ be a monomial in indeterminates $A_{r,n}^1\in\mathbb{T}_p(n^{-1/2}A_n^1(t))$, $r=1,2,\dots,p$, $p\in\mathbb{N}$. If the limits in (\ref{limcond2}) exist then 
	\begin{equation}\label{thm43re1}
		\lim_{n\rightarrow\infty}\frac{\mathbb{E}\mathrm{Trace}(A_{r_1,n}^1\dots A_{r_l,n}^1)}{n}=\begin{cases}
			0,\ l=2m-1,\\
			\displaystyle\prod_{\substack{j=1
            }}^{m}(t_{r_j}-t_{r_j-1})\sum_{w\in\mathcal{P}_2(2m,g)}\lim_{n\rightarrow\infty}\frac{\#\Pi(\chi_g(w))}{n^{1+m}},\ l=2m
		\end{cases}
	\end{equation}
    for all $m\ge1$.
\end{theorem}
\begin{proof}
	For $0\leq s<t<\infty$, using the stationary increments property of the Poisson process,  
    \begin{align*}
		\mathbb{E}|\mathcal{S}_{L^d(i,j),n}(t)-\mathcal{S}_{L^d(i,j),n}(s)|^l
		&=\sum_{h\leq r=0}^{\infty}\mathbb{E}\Big|\sum_{k=h}^{r}X_{L^d(i,j),k}\Big|^l\mathbb{P}\{N(ns)=h,\,N(nt)=r\}\nonumber\\
		&\leq \sum_{h\leq r=0}^{\infty}K_{l}(r-h)^{[l/2]}\mathbb{P}\{N(ns)=h,\,N(nt)=r\}\nonumber\\
		&= K_{l}\sum_{r'=0}^{[l/2]}S([l/2],r')(n(t-s))^{r'}\leq K_l'n^{l/2}(t-s)^{l/2}
	\end{align*}
     for every $l\ge3$ and for all $1\leq i,j\leq n$, and $\mathbb{E}(\mathcal{S}_{L^d(i,j),n}(t)-\mathcal{S}_{L^d(i,j),n}(s))^2=n(t-s)$. Here, we have used the $l$th moment of the Poisson process, that is, $\mathbb{E}(N(t))^l=\sum_{r=0}^{l}S(l,r)t^r$, to get the last step, and where $S(l,r)$ is the Stirling number of second kind. Now, the proof follows similar lines to that of Theorem \ref{thm3} (i).
\end{proof}
\begin{remark}\label{rem41}
	For $\alpha=1$, 
    $\mathcal{S}^1_{i,j,n}(t)=\sum_{k=1}^{N(nt)}X_{i,j,k}$ and $\mathcal{S}^1_{i,n}(t)=\sum_{k=1}^{N(nt)}X_{i,k}$, where $\{N(t)\}_{t\ge0}$ is a homogeneous Poisson process independent of $\{X_{i,j,k}\}$ and $\{X_{i,k}\}$, and $A^1_n(t)=n^{-1/2}(\mathcal{S}^1_{L^d(i,j),n}(t))_{1\leq i,j\leq n}$. From Theorem \ref{thm43}, for $p=1$ and $t_0=0$, it follows by arguments similar to those in the proof of Theorem \ref{thm2as}, that the ESD of $n^{-1/2}A^1_n(t)$ converges weakly a.s.~to a probability distribution whose $l$th moment is given by 
	\begin{equation}\label{rem41re1}
		\lim_{n\rightarrow\infty}\frac{\mathbb{E}\mathrm{Trace}(A_n^1(t))^l}{n^{1+l/2}}=\begin{cases}
			0,\ l=2m-1,\ m\ge1,\vspace{0.2cm}\\
			\displaystyle t^{m}\sum_{w\in\mathcal{P}_2(2m)}\lim_{n\rightarrow\infty}\frac{\#\Pi(w)}{n^{1+m}},\ l=2m,\ m\ge1.
		\end{cases}
	\end{equation}
	Thus, from (\ref{thm1re1}), the LSD of $n^{-1/2}A^1_n(t)$ and of $n^{-1}A_n(t)$ (as in Theorem \ref{thm1}) are identical.
\end{remark}
\begin{figure}[ht!]
    \centering
    \includegraphics[width=0.9\linewidth]{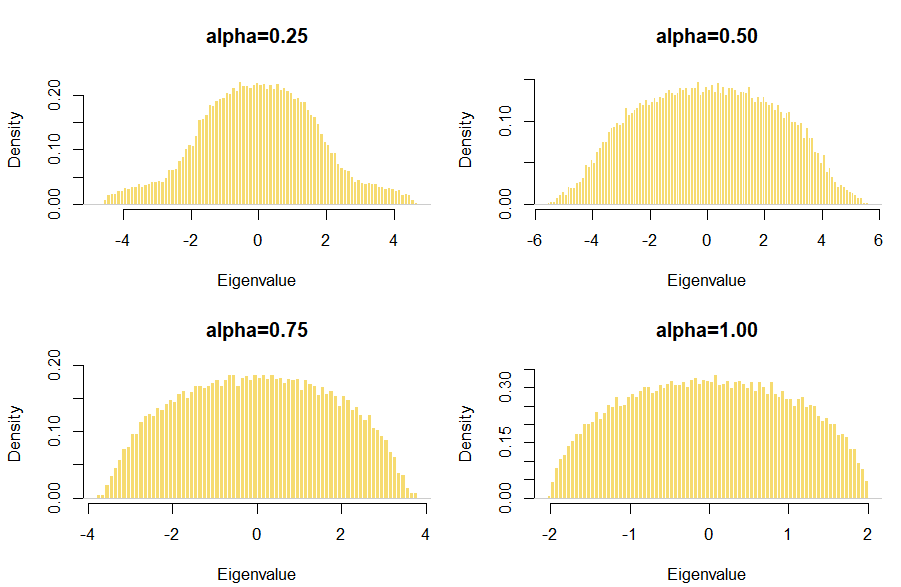}
    \caption{\small ESD of $n^{-1/2}A_n^\alpha(t)$ for Wigner type link function, $n=1000$, $t=1$, and with standard Gaussian step variables $X_{i,j,k}$'s.}\label{fig1}
\end{figure}

\begin{figure}[ht!]
    \centering
    \includegraphics[width=0.9\linewidth]{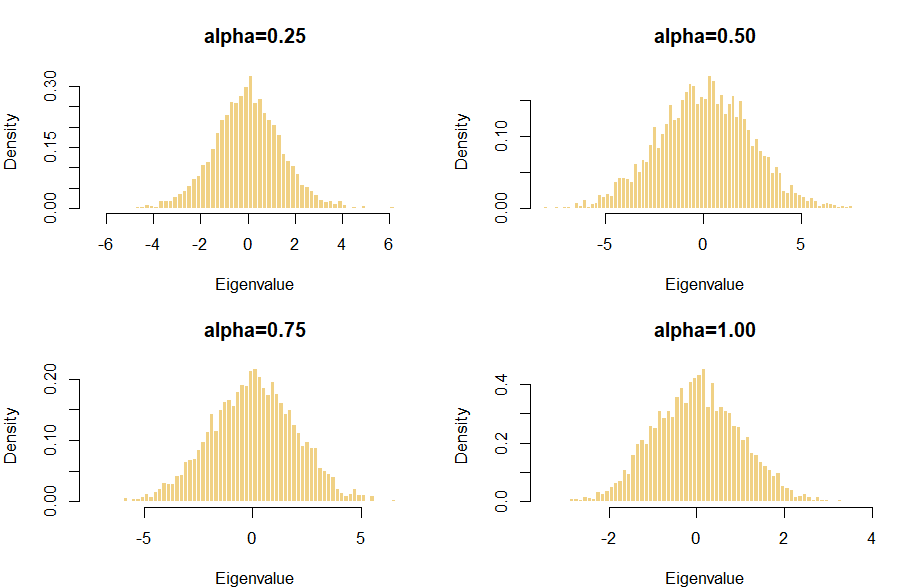}
    \caption{ESD of $n^{-1/2}A_n^\alpha(t)$ for Symmetric Circulant type link function, $n=1500$, $t=1$, and with standard Gaussian step variables $X_{i,j,k}$'s.}\label{fig2}
\end{figure}

\begin{remark}\label{rem42c}
    If the limits (\ref{limcond2}) exist for every monomial of even order then $\{n^{-1/2}A_n^1(t)\}_{t\ge0}$ associated to $(\mathcal{M}_n(\mathbb{C}),\ \tau_n=n^{-1}\mathbb{E}\mathrm{Trace})$ algebraically converges to $\{A^1(t)\}_{t\ge0}$ associated to $(\tilde{\mathcal{A}}^1,\tilde{\tau}^1)$, where algebra $\tilde{\mathcal{A}}^1$  is generated by the arbitrary union $\cup_{p\in\mathbb{N}}\cup_{\mathbb{T}_p}\mathcal{A}^1(\mathbb{T}_p)$ with $\mathcal{A}^1(\mathbb{T}_p)$ being polynomial algebra over $\mathbb{C}$, generated by the collection $\{A^1(t): t\in\mathbb{T}_p\}$. The state $\tilde{\tau}^1$ on it is defined similarly to (\ref{tildetau}). Indeed, $\tilde{\tau}^1(g(\{A^1(t): t\in\mathbb{T}_l\}))$ is a finite sum of limits of type (\ref{thm43re1}) and (\ref{rem41re1}).
	
	If $L^2(i,j)=(\max\{i,j\},\min\{i,j\})$, then the limiting variables $\{A^1(t)\}_{t\ge0}$ form an FBM. If $L^1(i,j)=n/2-|n/2-|i-j||$, $1\leq i,j\leq n$, then the finite dimensional moments of the limit variables agree with those of a BM.
\end{remark}
\begin{remark}\label{rem53}
	Let $N(n)$ be a Poisson rv with mean $n$. Let $\{A_{k,n}\}_{k\ge1}$ be a sequence of $n\times n$ independent random matrices defined as $A_{k,n}=n^{-1/2}((X_{L^d(i,j),k}))_{1\leq i,j\leq n}$ for all $k\ge1$, where the input sequence $X_{i,j,k}$'s or $X_{i,k}$'s are as in Theorem \ref{thm1}. Also, let $\{A_{k,n}\}_{k\ge1}$ be independent of $N(n)$, and have a common symmetric link function, satisfying Property B. If the limits (\ref{classlim}) exist, then the LSD of $n^{-1/2}(A_{1,n}+\dots +A_{N(n),n})$ coincides with that of $A_{k,n}$ for each $k\ge1$. It has a probability law whose $l$th moment equals (\ref{rem41re1}). Thus, from Remark \ref{rem31}, it follows that $n^{-1/2}(A_{1,n}+\dots +A_{N(n),n})$ and $n^{-1/2}(A_{1,n}+\dots+A_{n,n})$ have a common LSD. 
\end{remark}

Let $\{X_{i,j,k}, X_{i,k},\  i\ge0,\ j,\,k\ge1\}$ be as in Theorem \ref{thm1}. Let us consider 
\[S_{i,j,n}(t)=n^{-1/2}\sum_{k=1}^{N(nt)}X_{i,j,k}, \ 1\leq i,j\leq n, \ t\ge0,
\]and a random matrix 
\[A_n(t)=((S_{L^d(i,j),n}(t)))_{1\leq i, j \leq n}.\] Then, it can be easily established that Theorem \ref{thm43} holds true for the matrix valued process $\{A_{n^\alpha}(t)\}_{t\ge0}$, $\alpha\in(0,1]$ for any appropriate link function $L^d$. We now use this fact and the next result to establish the algebraic convergence of  $\{A^\alpha_n(t)\}_{t\ge0}$ (see 
(\ref{ctrm2}) for definition) to the time-changed FBM $\{B^\alpha(t)\}_{t\ge0}$ as defined in (\ref{tcfbm}).

\begin{lemma}\label{lemma1}
    Let $\{N(t)\}_{t\ge0}$ and $\{N^\alpha(t)\}_{t\ge0}$ be the homogeneous and fractional Poisson processes, respectively. Let $\{X_{k}\}_{k\ge1}$ be a sequence of independent rvs that is independent of $\{N(t)\}_{t\ge0}$ and $\{N^\alpha(t)\}_{t\ge0}$. Also, let $\{L^\alpha(t)\}_{t\ge0}$ be an inverse $\alpha$-stable subordinator that is independent of both $\{N(t)\}_{t\ge0}$ and $\{X_{k}\}_{k\ge1}$. If $S_{n}(t)=n^{-1/2}\sum_{k=1}^{N(nt)}X_{k}$ and $S_{n}^\alpha(t)=n^{-\alpha/2}\sum_{k=1}^{N^\alpha(nt)}X_{k}$ for all $t\ge0$, then $S^\alpha_n(t)\overset{d}{=}S_{n^\alpha}(L^\alpha(t))$, where $\overset{d}{=}$ denotes the equality in distribution.
\end{lemma}  
\begin{proof}
 From \cite{Meerscheart2013}, it follows that the inverse $\alpha$-stable subordinator is a self-similar process of index $\alpha\in(0,1)$, that is, $L^\alpha(ct)\overset{d}{=}c^\alpha L^\alpha(t)$ for $c>0$, and from (\ref{tcr}), we have $N^\alpha(nt)=N(L^\alpha(nt))$. Therefore, $N^\alpha(nt)\overset{d}{=}N(n^\alpha L^\alpha(t))$ for all $t\ge0$. This completes the proof.
\end{proof}

\begin{theorem}\label{thm42c}
Let $\{A^\alpha_n(t)\}_{t\ge0}$ be a matrix valued process where $A^\alpha_n(t)$ is defined in (\ref{ctrm2}). For $p\ge1$, $\mathbb{T}_p(n^{-1/2}A^\alpha(t))$ be the collection of increments of $\{A^\alpha_n(t)\}_{t\ge0}$, as defined in (\ref{incset}). Let us consider a monomial $g(\mathbb{T}_p(n^{-1/2}A^\alpha(t)))=A^\alpha_{r_1,n}\dots A^\alpha_{r_l,n}$, where $A_{r_j,n}^\alpha\in\mathbb{T}_p(n^{-1/2}A^\alpha(t))$ for each $j=1,2,\dots,l\ge1$. Also, let $F^{\alpha}_{t_1, \ldots t_p}(x_1, \ldots, x_p)$ be the finite dimensional distribution of an inverse $\alpha$-stable subordinator $\{L^\alpha(t)\}_{t\ge0}$, that is, $F^{\alpha}_{t_1, \ldots t_p}(x_1, \ldots, x_p)=\mathbb{P}\{L^\alpha(t_1)\leq x_1,\dots,L^\alpha(t_p)\leq x_p\}$, $x_1,\dots,x_p\in[0,\infty)$. If the limits in (\ref{limcond2}) exist, then
	\begin{align}\label{thm44re1}
		\lim_{n\rightarrow\infty}&\frac{\mathbb{E}\mathrm{Trace}(A_{r_1,n}^\alpha\dots A_{r_l,n}^\alpha)}{n}\nonumber\\
        &=\begin{cases}
			0,\ l=2m-1,\ m\ge1,\\			\displaystyle
            \iint_{[0,\infty)^{2m}}
            \prod_{\substack{j=1
            }}^{m}(x_{r_j}-x_{r_j-1})\sum_{w\in\mathcal{P}_2(2m,g)}\lim_{n\rightarrow\infty}\frac{\#\Pi(\chi_g(w))}{n^{1+m}}\\
            \hspace{1cm}\cdot \mathrm{d}F^\alpha_{t_{r_1-1},t_{r_1},\dots,t_{r_m-1},t_{r_m}}(x_{r_1-1},x_{r_1},\dots,x_{r_m-1},x_{r_m}),\ l=2m,\,m\ge1.
		\end{cases}
	\end{align}
\end{theorem} 
\begin{proof}
    For a circuit $\pi$ of length $l\ge1$, set
    {\small\begin{equation*}
        S_{\pi,n}^\alpha=\mathbb{E}(S_{L^d(\pi(0),\pi(1)),n}^\alpha(t_{r_1})-S_{L^d(\pi(0),\pi(1)),n}^\alpha(t_{r_1-1}))\dots(S_{L^d(\pi(l-1),\pi(l)),n}^\alpha(t_{r_l})-S_{L^d(\pi(l-1),\pi(l)),n}^\alpha(t_{r_l-1})).
    \end{equation*}}
    By abuse of notation, let $\{r_1,\dots,r_{l'}\}$, $1\leq 1\leq l$ be the collection of all distinct $r_j$'s, and let $\{S_{i,j,n}(t)=n^{-1/2}\sum_{k=1}^{N(nt)}X_{i,j,k}\}_{t\ge0}$ be independent of $\{L^\alpha(t)\}_{t\ge0}$. Then, from Lemma \ref{lemma1}, we have
    \begin{equation*}
        S_{\pi,n}^\alpha        =\iint_{\mathbb{R}^{2l'}_+}S_{\pi, n^\alpha}\mathrm{d}F^\alpha_{t_{r_1-1},t_{r_1},\dots,t_{r_{l'}-1},t_{r_{l'}}}(x_{r_1-1},x_{r_1}\dots,x_{r_{l'}-1},x_{r_{l'}}),
    \end{equation*}
    where 
    {\footnotesize\begin{equation*}
        S_{\pi, n^\alpha}=\mathbb{E}(S_{L^d(\pi(0),\pi(1)),n^\alpha}(x_{r_1})-S_{L^d(\pi(0),\pi(1)),n^\alpha}(x_{r_1-1}))\dots(S_{L^d(\pi(l-1),\pi(l)),n^\alpha}(x_{r_l})-S_{L^d(\pi(l-1),\pi(l)),n^\alpha}(x_{r_l-1})).
    \end{equation*}}
    As the inverse stable subordinator is an increasing process (see \cite{Meerscheart2013}), for distinct indexes $r_j$'s, the rvs $(S_{L^d(\pi(0),\pi(1)),n^\alpha}(x_{r_1})-S_{L^d(\pi(0),\pi(1)),n^\alpha}(x_{r_1-1})),\dots,(S_{L^d(\pi(l-1),\pi(l)),n^\alpha}(x_{r_l})-S_{L^d(\pi(l-1),\pi(l)),n^\alpha}(x_{r_l-1}))$ are mutually independent due to the independent increments property of the process $\{S_{i,j,n^\alpha}(t)=n^{-\alpha/2}\sum_{k=1}^{N(n^\alpha t)}X_{i,j,k}\}_{t\ge0}$. Thus, 
    \begin{align}
        \lim_{n\rightarrow\infty}\frac{\mathbb{E}\mathrm{Trace}(A_{r_1,n}^\alpha\dots A_{r_l,n}^\alpha)}{n}
        &=\lim_{n\longrightarrow\infty}\frac{1}{n^{1+l/2}}\sum_{\pi:\pi \text{is a circuit of length}\ l}\iint_{\mathbb{R}^{2l'}_+}S_{\pi, n^\alpha}\nonumber\\
        &\hspace{0.5cm}\cdot\mathrm{d}F^\alpha_{t_{r_1-1},t_{r_1},\dots,t_{r_{l'}-1},t_{r_{l'}}}(x_{r_1-1},x_{r_1}\dots,x_{r_{l'}-1},x_{r_{l'}}).\label{thm44pf1}
    \end{align}
    Note that for $l=2m$ in (\ref{thm44pf1}), we have $l'=m$. Similar to Theorem \ref{thm43}, it can be shown that $\sum_{\pi:\pi \text{is a circuit of length}\ l}\lim_{n\longrightarrow\infty}{S_{\pi, n^\alpha}}/{n^{1+l/2}} $ equals (\ref{thm43re1}), 
    which on using in (\ref{thm44pf1}), we get the required result. 
    Also, the interchange of integral and limit in (\ref{thm44pf1}) is justified due to the dominated convergence theorem because for each $\pi$, $S_{\pi,n^\alpha}/n^{1+(1+\alpha)l/2}$ can be bounded by an appropriate function of $x_{r_1-1}, x_{r_1},\dots,x_{r_l-1},x_{r_l}$, and all the finite dimensional joint moments of an inverse stable subordinator exist (see \cite{Veillette2010}). 
\end{proof} 
\begin{remark}\label{rem54}
    For $p=1$, $t_0=0$, from (\ref{thm44re1}), we get $\lim_{n\rightarrow\infty}n^{-1}{\mathbb{E}\mathrm{Trace}(A_{n}^\alpha(t)^{l}}=0$ whenever $l$ is odd. For $m\ge1$, it  reduces to 
    \begin{align}
        \lim_{n\rightarrow\infty}\frac{\mathbb{E}\mathrm{Trace}(A_{n}^\alpha(t)^{2m}}{n}&=\int_{0}^{\infty}x^m\sum_{w\in\mathcal{P}_2(2m,g)}\lim_{n\rightarrow\infty}\frac{\#\Pi(\chi_g(w))}{n^{1+m}}\mathbb{P}\{L^\alpha(t)\in\mathrm{d}x\}\nonumber\\
        &=\frac{t^{m\alpha}\Gamma(m+1)}{\Gamma(m\alpha+1)}\sum_{w\in\mathcal{P}_2(2m,g)}\lim_{n\rightarrow\infty}\frac{\#\Pi(\chi_g(w))}{n^{1+m}},\label{thm41re1}
    \end{align}
    where we have used (\ref{lmomis}). Thus, the EESD of $n^{-1/2}A_n^\alpha(t)$ weakly converges to a probability law whose all odd moments are zero and $2m$th moment equals (\ref{thm41re1}). Moreover, following the proof of Theorem \ref{thm2as}, 
it can be easily shown that the ESD of $n^{-1/2}A^\alpha_n(t)$ converges weakly a.s.~to a probability law with zero odd moments and even moments are given by (\ref{thm41re1}). In particular, the self-adjoint element $n^{-1/2}A_n^\alpha(t)$ of $(\mathcal{M}_n(\mathbb{C}),n^{-1}\mathbb{E}\mathrm{Trace})$ converges to an element $A^\alpha(t)$ of an algebra generated by $A^\alpha(t)$ with a state defined by $\tau((A^\alpha(t))^l)=0$ for all odd $l$ and $\tau((A^\alpha(t))^{2m})$ equals (\ref{thm41re1}) for each $m\ge1$.
\end{remark}

\begin{remark}\label{rem43}
	Let $\{A_{k,n}\}_{k\ge1}$ be a sequence of random matrices as defined in Remark \ref{rem53}, and it is independent of the fractional Poisson rv $N^\alpha(n)$. If the limits (\ref{classlim}) exist, then the ESD of $A_{k,n}$ converges weakly a.s.~to a probability distribution whose $l$th moment is given by (\ref{rem41re1}) evaluated at $t=1$. Moreover, the scaled random sum $n^{-\alpha/2}(A_{1,n}+\dots+A_{N^\alpha(n),n})$ has an LSD whose $l$th moment is given by (\ref{thm41re1}) at $t=1$.
\end{remark}

 The following result follows from Remark \ref{rem54} and the results of \cite{Bose2008}:
\begin{corollary}\label{cor4.1}
   For $t\ge0$, let $A^\alpha_n(t)$ be the random matrix as defined in (\ref{ctrm2}). Then,\vskip5pt
   
   \noindent (i) for $L^2(i,j)=(\max\{i,j\},\min\{i,j\})$, the matrix $n^{-1/2}A^\alpha_n(t)$ in $(\mathcal{M}_n(\mathbb{C}),\tau_n=n^{-1}\mathbb{E}\mathrm{Trace})$ converges to $B^\alpha(t)$ (for definition see (\ref{tcfbm})) in polynomial algebra $\mathbb{C}\langle B^\alpha(t)\rangle$ with state defined as $\tau((B^\alpha(t))^l)=\lim_{n\rightarrow\infty}\tau_n((A^\alpha_n(t))^l)$ for all $l\ge1$;\vskip3pt
   
   \noindent (ii) If $L^1(i,j)=n/2-|n/2-|i-j||$ for $1\leq i,j\leq n$ then $n^{-1/2}A^\alpha_n(t)$ algebraically converges to a variable whose moments 
   determine the probability law same as that of $W(L^\alpha(t))$, where $\{W(t)\}_{t\ge0}$ is a BM that is independent of an inverse $\alpha$-stable subordinator $\{L^\alpha(t)\}_{t\ge0}$.  
\end{corollary}

\subsection{Symmetric circulant and weak convergence to a time-changed BM}\label{symcirctimechangedsbm} For $\alpha\in(0,1]$, let $\lambda_{1,n}^\alpha(t),\dots,\lambda_{n,n}^\alpha(t)$ be the eigenvalues of the random matrix $A_n^\alpha(t)$, as defined in (\ref{ctrm2}) with link function $L^1(i,j)=n/2-|n/2-|i-j||$, $1\leq i,j\leq n$. Let $U_n$ be a uniform rv on $\{1,\dots,n\}$. We consider a random process $\{Y_n^\alpha(t)\}_{t\ge0}$ defined as
\begin{equation}\label{defctcbm}
    Y_n^\alpha(t)\coloneqq \lambda_{U_n,n}^\alpha(t),\ \text{for all $t\ge0$}
\end{equation}
In particular, we have 
\begin{equation*}
    \lambda_{r,n}^\alpha(t)=n^{-\alpha/2}S_{0,n}^\alpha(t)+2n^{-\alpha/2}\sum_{j=1}^{[n/2]}S_{j,n}^\alpha(t)\cos\Big(\frac{2\pi jr}{n}\Big),\  0<\alpha\leq1,\ r=1,2,\dots,n,
\end{equation*}
where $S^\alpha_{j,n}(t)$ is as defined in (\ref{ctrw22}). For $\alpha=1$, it reduces to
\begin{equation*}
    \lambda_{r,n}^1(t)=n^{-1/2}S_{0,n}(t)+2n^{-1/2}\sum_{j=1}^{[n/2]}S_{j,n}(t)\cos\Big(\frac{2\pi jr}{n}\Big),\ r=1,2,\dots,n,
\end{equation*}
where $S_{j,n}(t)=\sum_{k=1}^{N(nt)}X_{j,k}$ is as defined in  Remark \ref{rem41}, and $\{N(t)\}_{t\ge0}$ is a Poisson process with parameter one and it is independent of $X_{j,k}$'s. Let $\{L^\alpha(t)\}_{t\ge0}$ be an inverse $\alpha$-stable subordinator (see Section \ref{fpp}) that is independent of $\{N(t)\}_{t\ge0}$ and $X_{j,k}$'s. From Lemma \ref{lemma1}, we get $\lambda^\alpha_{r,n}(t)\overset{d}{=}\lambda_{r,n^\alpha}^1(L^\alpha(t))$ for all $t\ge0$. Thus,
\begin{equation}\label{tcr2}
    Y^\alpha_{n}(t)\overset{d}{=}Y^1_{n^\alpha}(L^\alpha(t)),\ t\ge0.
\end{equation}
\begin{theorem}\label{thm4.3}
Let $Y^\alpha_n(t)$ be as defined in (\ref{defctcbm}). Then, the process $\{n^{-1/2}Y_n^\alpha(t)\}_{t\geq 0}$ converges weakly to $\{W(L^\alpha(t))\}_{t\geq 0}$, where $\{W(t)\}_{t\ge0}$ is a BM that is independent of the inverse $\alpha$-stable subordinator $\{L^\alpha(t)\}_{t\ge0}$.
\end{theorem}
\begin{proof}
  In view of Remark \ref{rem42c}, following the proof of Theorem \ref{thmwc1}, it is easy to show that $\{n^{-1/2}Y^1_{n^\alpha}(t)\}_{t\geq 0}$ converges weakly to a BM $\{W(t)\}_{t\geq 0}$. So, $\{(n^{-1/2}Y^1_{n^\alpha}(t), L^\alpha(t))\}_{t\geq 0}$ jointly weakly converges to $\{(W(t),L^\alpha(t))\}_{t\geq 0}$. Thus, from (\ref{tcr2}) and continuous mapping theorem, $\{n^{-1/2}Y_n^\alpha(t)\}_{t\geq 0}$ converges to $\{W(L^\alpha(t))\}_{t\geq 0}$ in an appropriate Skorokhod topology (for details see \cite{Meerschaert2012}, Section 4.4). This completes the proof.
  \end{proof}

  \section{Non-symmetric patterned random matrices}\label{nonsymmetric} So far we have dealt with patterned matrices that have a symmetric link function. Now we consider matrices with link functions that are not necessarily symmetric. Then the matrix polynomials are formed using the matrices along with their transposes. 
  
  The joint algebraic convergence (convergence of the $*$-distributions) of iid copies of these matrices can be discussed under assumptions similar to those we have imposed earlier. Specifically, the joint convergence of all the matrices considered in the next sections below holds.
The specific patterned matrices we consider are the general circulant matrix (the non-symmetric version of the symmetric circulant matrix), and the elliptic matrix (where the $(i,j)$th and the $(j,i)$th entries are correlated, a special case being the iid matrix, the non-symmetric version of the Wigner matrix). 
  Results similar to those derived earlier for the Wigner and the symmetric circulant are given. However, we skip the detailed proofs. 


\subsection{Circulant matrix and complex BM}   

The circulant matrix is defined as $C_n:=((a_{L^1(i,j)}))_{1\leq i, j \leq n}$ where 
\begin{equation}\label{circL}L^1(i,j):=(j-i+n)\ \text{mod}\,n, \ 1\leq i,j\leq n.
\end{equation}
It is a non-symmetric matrix and hence has complex eigenvalues. 
Theorem 3.2.1 of \cite{Bosesaha2018} states that if the entries are independent with mean $0$, variance $1$, and with uniformly bounded $(2+\delta)$th moment for some $\delta > 0$, then the ESD of $n^{-1/2}C_n$ 
converges in probability to the complex (bivariate)  Gaussian distribution. The authors used Gaussian approximation to prove this result.  Under Assumption I, this convergence can be shown to be almost sure. We omit the detailed proof of this.


Now let $C_{k,n}:=n^{-1/2}((a_{k, L^1(i,j)}))_{1\leq i, j \leq n}$, $1\leq k \leq p\in\mathbb{N}$ be $p$ independent $n\times n$ circulant matrices whose entries satsify Assumption I. Then for $1\leq k \leq p$ their eigenvalues are given by 
\begin{equation}\label{circeval}
    \lambda^c_{k, r}\coloneqq n^{-1/2}\sum_{j=0}^{n-1}a_{k,j}\cos\Big(\frac{2\pi rj}{n}\Big)+\iota n^{-1/2}\sum_{j=0}^{n-1}a_{k,j}\sin\Big(\frac{2\pi rj}{n}\Big),\ 0\leq r\leq n-1.
\end{equation}
 Define the following joint 
ESD:
\begin{equation}\label{circesd}
n^{-1}\sum_{r=0}^{n-1}\mathbb{I} \big(\text{Re}(\lambda^c_{k, r})\leq x_k, \text{Im} (\lambda^c_{k, r})\leq y_k), k=1, \ldots p\big). 
\end{equation}
Then, by an easy extension of the Gaussian approximation arguments (Berry-Essen bounds) given in Bose \cite{Bosesaha2018}, we have the following lemma. We omit its proof. We recall that a standard complex Gaussian random variable is of the form $N=(X+\iota Y)/\sqrt{2}$ where $X$ and $Y$ are iid standard Gaussian variables. 

\begin{lemma} The ESD (\ref{circesd}) converges weakly a.s.~to (the distribution of) $(N_1, \ldots, N_p)$ where $\{N_i\}$ are iid complex (considered as elements of $\mathbb{R}^2$) standard Gaussian random variables. 
\end{lemma}
Let us now turn to joint moment convergence. 
Note that Theorem 9.5.1 of \cite{Bose2021} shows that independent scaled \textit{symmetric} circulant matrices converge to independent Gaussian variables algebraically. We need a similar result for circulant matrices. By adapting the arguments given there and those given earlier in Section \ref{jtcov}, it is not too difficult to show that only pair-matched words survive when we consider the moment of any monomial, but it does not seem easy to identify/characterize the algebraic limit directly. On the other hand, Lemma \ref{circesd} suggests this algebraic limit must also be $(N_1, \ldots, N_p)$. 
Indeed we have the following result. 

\begin{lemma}\label{lemma51}
    Let $\{N_k=N_{k,1}+iN_{k,2}\}_{1\leq k\leq p}$, $p\in\mathbb{N}$ be independent complex Gaussian rvs.
    Let $\{n^{-1/2}C_{k,n}\}_{1\leq k\leq p}$ be $n\times n$ independent circulant matrices with input sequences satisfying Assumption I. Then, as elements of $(\mathcal{M}_n(\mathbb{C}), \tau_n=n^{-1}\mathbb{E}\mathrm{Trace}(\cdot))$, they 
    converge in $*$-distribution to $\{N_1, \ldots , N_p\}$. That is,
    \begin{equation*}
        \lim_{n\rightarrow\infty}\tau_{n}(g(\{n^{-1/2}C_{k,n},n^{-1/2}C_{k,n}^T\}_{1\leq k\leq p}))=\mathbb{E}(g(\{N_k,\bar{N}_k\}_{1\leq k\leq p}))
    \end{equation*} 
    for every polynomial $g$. Here, $C^T$ and $\bar{N}$ denote the transpose and complex conjugate of $C$ and $N$, respectively.
\end{lemma} 

\begin{proof} 
We present only an outline of the proof.
As circulant matrices commute, any typical monomial in indeterminates $\{C_{k,n}\}_{1\leq k\leq p}$ and $\{C_{k,n}^T\}_{1\leq k\leq p}$ is of the following form: $C_{1,n}^{m_1}(C_{1,n}^T)^{l_1}\cdots C_{p,n}^{m_p}(C_{p,n}^T)^{l_p}$ for some $\{m_1,\dots,m_p,l_1,\dots,l_p\}\subset\mathbb{N}\cup\{0\}$. Note that the circulant matrices are simultaneously diagonalizable using the discrete Fourier matrix. Hence,
\begin{equation*}
  n^{-1}\mathbb{E}  \mathrm{Trace}  \Big(
    (\dfrac{C_{1,n}}{n^{1/2}})^{m_1}\big(\dfrac{C_{1,n}^T}{n^{1/2}}\big)^{l_1}\cdots \big(\dfrac{C_{p,n}}{n^{1/2}}\big)^{m_p}\big(\dfrac{C_{p,n}^T}{n^{1/2}}\big)^{l_p}\Big)=n^{-1}\sum_{r=0}^{n-1}\mathbb{E}\Big(\lambda_{1,r}^{m_1}\bar{\lambda}_{1,r}^{l_1}\dots \lambda_{p,r}^{m_p}\bar{\lambda}_{p,r}^{l_p}\Big),
\end{equation*}
where $\lambda_{k,0},\dots,\lambda_{k,n-1}$ are eigenvalues of $n^{-1/2}C_{k,n}$ for each $1\leq k\leq p$. 
By independence of the matrices, this expectation equals
\begin{equation}\label{eq:traceofprodcirc}
    n^{-1}\sum_{r=0}^{n-1}\mathbb{E}\Big(\lambda_{1,r}^{m_1}\bar{\lambda}_{1,r}^{l_1}\Big)\cdots \mathbb{E}\Big(\lambda_{p,r}^{m_p}\bar{\lambda}_{p,r}^{l_p}\Big).
    \end{equation}
Consider each factor in the product. 
By the Gaussian approximation arguments of \cite{Bosesaha2018}, we know that for every fixed $r$, 
$(\lambda_{l,r}, 1\leq l \leq p)$ converges weakly to $(N_l, \ldots , N_p)$. 
Moreover, using Assumption I, any power of these variables is uniformly integrable, and hence their expectations also converge to the corresponding expectations of the limit variables, for every fixed $r$. Indeed, by using an appropriate non-uniform Berry-Esseen bound, this convergence is also uniform over $0\leq r\leq n-1$. That is, for every $\{m_k, l_k\}$, 
\[\sup_{\substack{1\leq k \leq p,\\0\leq r\leq n-1}}
\Big|\mathbb{E}\big(\lambda_{k,r}^{m_k}\bar{\lambda}_{k,r}^{l_k}\big)-
\mathbb{E}\big(N_{k}^{m_k}\bar N_{k}^{l_k}\big)\Big|\to 0\ \ \text{as}\ \ n\to \infty.
\]
Since (\ref{eq:traceofprodcirc}) is an average over $0\leq r\leq n-1$, the limit in (\ref{eq:traceofprodcirc}) equals $\prod_{k=1}^p\mathbb{E}\big(N_{k}^{m_k}\bar N_{k}^{l_k}\big)$.
This proves the result. 
\end{proof}


Now, with $L^1$,  $S_{i,j}^{(n)}(t)$, and $S_{i,j,n}^\alpha(t)$ as in (\ref{circL}),  (\ref{ctrw1}) and (\ref{ctrw2}), respectively,
define the circulant matrices  with random walk and stopped random walk entries as follows:
\begin{align}
    C_n(t)&\coloneqq((S_{L^1(i,j)}^{(n)}(t)))_{1\leq i,j\leq n},\\
    C_n^\alpha(t)&\coloneqq((S_{L^1(i,j),n}^\alpha(t)))_{1\leq i,j\leq n},\ 0<\alpha\leq 1,\ t\ge0,
    \end{align}
respectively. Then, in view of Lemma \ref{lemma51}, similar to Theorem \ref{thm3} (ii) and Corollary \ref{cor4.1} (ii),  we have the following result. We need the (standard) complex Brownian motion (CBM). We can write it as 
$CBM(t)\stackrel{d}{=} (B_1+\iota B_2)/\sqrt{2}$,  where $B_1$ and $B_2$ are independent BMs and $\iota:=\sqrt{-1}$. Its time changed variant is defined in the obvious way and will be denoted by $CBM^\alpha(\cdot)$.
We omit the proof of the next result. 

\begin{theorem}
    Under Assumption I, $\{n^{-1}C_n(t)\}_{t\ge0}$ and $\{n^{-(1+\alpha)/2}C_n^\alpha(t)\}_{t\ge0}$ converge algebraically (that is, in their $*$-distribution) respectively to $\{CBM(t)\}_{t\ge0}$, and $\{CBM^\alpha (t)\}_{t\ge0}$. 
\end{theorem}

Using (\ref{circeval}), the eigenvalues of $C_n(t)$ and $C_n^\alpha(t)$ are given respectively by 

\begin{align} 
\beta_{r,n}&\coloneqq\sum_{j=0}^{n-1}S_{j,n}(t)\cos\Big(\frac{2\pi rj}{n}\Big)+\iota\sum_{j=0}^{n-1}S_{j,n}(t)\sin\Big(\frac{2\pi rj}{n}\Big),\ 0\leq r\leq n-1,\\
\beta^\alpha_{r,n}&\coloneqq\sum_{j=0}^{n-1}S_{j,n}^\alpha(t)\cos\Big(\frac{2\pi rj}{n}\bigg)+\iota\sum_{j=0}^{n-1}S_{j,n}^\alpha(t)\sin\Big(\frac{2\pi rj}{n}\bigg),\ 0\leq r\leq n-1.
\end{align}

Let $U_n$ be a rv independent of everything else and uniformly distributed on $\{0, 2, \ldots, n-1\}$. Similar to   $Y_n(t)$ and $Y_n^\alpha(t)$ defined in (\ref{newp}) and (\ref{defctcbm}), define the empirical processes  
\begin{align*}Z_n(t)&:=\beta_{U_{n},n}(t), \ \text{and} \ \ Z_n^\alpha(t):=\beta^\alpha_{U_{n},n}(t).
\end{align*} Then similar to Theorems \ref{thmwc1} and \ref{thm4.3}, we have the following result. We omit the proof. 

\begin{theorem} Under Assumption I, $n^{-1}Z_n(\cdot)$ and $n^{-(1+\alpha)/2}Z_n^\alpha(\cdot)$ converge weakly respectively to $CBM(\cdot)$ and $CBM^\alpha(\cdot)$ in suitable Skorokhod topologies.
\end{theorem}

\subsection{Elliptic matrices with CTRW entries} 
   An elliptic variable $e$ with correlation parameter $-1\leq \rho\leq 1$ can always be described as (see page 38 of \cite{Bose2021})
\[e=\sqrt{\dfrac{1+\rho}{2}} s_1+\iota \sqrt{\dfrac{1-\rho}{2}}s_2,\] where $s_1$ and $s_2$ are two free standard semi-circular variables. A special case is when $\rho=0$ and in that case we get the standard circular variable, $c\stackrel{d}{=} (s_1+\iota s_2)/\sqrt 2$, where $s_1$ and $s_2$ are two free standard semi-circular variables. 

Suppose $B_1(t)_{t\geq 0}$ and $B_2(t)_{t\geq 0}$ are free FBMs and are free of each other. Let 
\[E(t):=\sqrt{\dfrac{1+\rho}{2}} B_1(t)+\iota \sqrt{\dfrac{1-\rho}{2}}B_2(t), \ t\geq 0.\]
Then $E(\cdot)$ has free independent increments, $E(t)=0$, and for every $ s< t$, $(E(t)-E(s))/\sqrt{t-s}$ is an elliptic variable with parameter $\rho$.  
The stopped version of $E(\cdot)$ is defined in the obvious way and will be denoted by $E^\alpha(\cdot)$. 

Suppose the pairs $(X_{i,j}, X_{j,i})$ are independent, each component has mean $0$, variance $1$, and all variables have uniformly bounded moments, and the pairs $(X_{i,j}, X_{j, i})$ have a common correlation $\rho$ for all $i\neq j$. Then the non-symmetric random matrix  $((X_{i,j}))_{1\leq i, j\leq n}$ is called an elliptic matrix.  The special case $\rho=0$ includes the fully independent matrix whose all elements are independent. The special case $\rho=1$ includes the Wigner matrix, which is symmetric. It is known that independent scaled elliptic matrices (even with different $\rho$'s) converge in $*$-distribution to free elliptic variables with the corresponding $\rho$'s (see Theorem 11.1.1 of \cite{Bose2021}). 

Now consider the following assumption which is similar to Assumption I but incorporates dependent variables, though in a limited way. \vskip5pt

\noindent \textbf{Assumption} II \ $\{X_{i,j,k},\ i\ge1,\ j\ge1,\ k\ge1\}$ are rvs such that the pairs $(X_{i,j,k}, X_{j, i,k})$ are independent, have a common  correlation $\rho$, and each component has mean $0$, variance $1$, with all moments uniformly bounded.\vskip5pt

\noindent Using the result cited above, and going through the arguments given earlier for symmetric matrices with random walk entries, we have the following theorem. 
We omit its proof.  
 
\begin{theorem} Consider the elliptic matrices 
$E_n(t):=((S_{i,j}^{(n)}(t)))_{1\leq i,j \leq n}, \ t\geq 0$, where $S_{i,j}^{(n)}(t)$ are as defined in (\ref{ctrw1}) and $\{X_{i,j,k}\}$ satisfy Assumption II. 
 Then $\{n^{-1}E_n(t)\}_{t \geq 0}$ converges in $*$-distribution to $\{E(t)\}_{t \geq 0}$. The stopped version  $n^{-(1+\alpha)/2}E_n^\alpha(\cdot)$ of  $n^{-1}E_n(\cdot)$ converges in $*$-distribution to $E^\alpha(\cdot)$. As a consequence, the ESD of any self-adjoint polynomial in $\{n^{-1}E_n(t)\}_{t \geq 0}$ or $\{n^{-(1+\alpha)/2}E_n^\alpha (t)\}_{t \geq 0}$ converges weakly a.s.~to the probability distribution of the corresponding polynomial in $E(\cdot)$ or $E^\alpha(\cdot)$ respectively.
 \end{theorem} 


\begin{thebibliography}{00}	
\bibitem{Anderson2009}
Anderson, G.W., Guionnet, A., Zetiouni, O. (2009). \textit{An Introduction to Random Matrices}. Cambridge University Press, Cambridge.
\bibitem{Bai1999}
Bai, Z.D. (1999). Methodologies in spectral analysis of large dimensional random matrices, a review. \textit{Stat. Sin.} \textbf{9}, 611-677.
\bibitem{Biane1998}
Biane, P. (1998). Processes with free increments. \textit{Math. Z.} \textbf{227}, 143-174.
\bibitem{Bian1997}
Biane, P. (1997). Free Brownian motion, free stochastic calculus and random matrices, in: \textit{Free Probability Theory}, Fields
Institute Communications, vol. 12, Amer. Math. Soc., Providence, RI, pp. 1-19.
\bibitem{Billingsley1999}
Billingsley, P. (1999). \textit{Convergence of Probability Measures}, Second Edition. Wiley Series in Statistics and Probability, John Wiley and Sons, Inc.
 \bibitem{Bryc2006}
 Bryc, W., Dembo, A., Jiang, T. (2006). Spectral measure of large random Hankel,
 Markov and Toeplitz matrices. \textit{Ann. Probab.} \textbf{34}(1), 1-38.
\bibitem{Beghin2009}
Beghin, L., Orsingher, E. (2009). Fractional Poisson processes and related planar random	motions. {\it Electon. J. Probab.} \textbf{14}(61), 1790-1826.
\bibitem{Banerjee2013}
Banerjee, S., Bose, A. (2013). Noncrossing partitions, Catalan words, and the semi-circle law. \textit{J. Theor. Probab.} \textbf{26}, 386-409. 
\bibitem{Bose2021}
Bose, A. (2021). \textit{Random Matrices and Non-Commutative Probability}. Chapman and Hall/CRC. 
\bibitem{Bose2008}
Bose, A., Sen, A. (2008). Another look at the moment method for large dimensional random matrices. \textit{Electron. J. Probab.} \textbf{13}(21), 588-628.
\bibitem{Bose2011}
 Bose, A., Hazra, R.S., Saha, K. (2011). Convergence of joint moments for independent random patterned matrices. \textit{Ann. Probab.} \textbf{39}(4), 1607-1620.
 \bibitem{Bose2018}
 Bose, A. (2018). \textit{Patterned Random Matrices}. Chapman and Hall/CRC.
 \bibitem{Bosesaha2018}
 Bose, A., Saha, K. (2018). \textit{Random Circulant Matrices}. Chapman and Hall/CRC.
 \bibitem{Dyson1962}
 Dyson, F.J. (1962). A Brownian-motion model for the eigenvalues of a random matrix. \textit{J. Math. Phys.} \textbf{3}, 1191-1198.
 \bibitem{Gouri2006}
  Gouri\'eroux, C. (2006). Continuous time Wishart process for stochastic risk. \textit{Econometric Rev.} \textbf{25}  177-217.
  \bibitem{Gnoatto2012}
   Gnoatto, A. (2012). The Wishart short rate model. \textit{Int. J. Theor. Appl. Finance} \textbf{15}, 1250056, 24.
   \bibitem{Gnoatto2014}
  Gnoatto, A., Grasselli, M. (2014). An affine multicurrency model with stochastic volatility and stochastic interest rates. \textit{SIAM J. Financial Math.} \textbf{5}, 493-531.
\bibitem{Kilbas2006}
Kilbas, A.A., Srivastava, H.M., Trujillo, J.J. (2006). \textit{Theory and Applications of Fractional Differential Equations}. Elsevier Science B.V., Amsterdam.
\bibitem{Laskin2003}
Laskin, N. (2003). Fractional Poisson process. \textit{Commun. Nonlinear Sci. Numer. Simul.} \textbf{8}, 201-213.
\bibitem{Li2016}
Li, J., Zhao, B., Deng, C., Xu, R.Y.D. (2016). Time varying metric learning for visual tracking. \textit{Pattern Recognit. Lett.} \textbf{80}, 157-164.
\bibitem{Mainardi2004}
Mainardi, F., Gorenflo, R., Scalas, E. (2004). A fractional generalization of the Poisson processes. \textit{Vietnam Journ. Math.} \textbf{32}, 53–64.
\bibitem{Meerschaert2011}
Meerschaert, M.M., Nane, E., Vellaisamy, P. (2011). The fractional Poisson process and the inverse stable subordinator. \textit{Electron. J. Probab.} \textbf{16}(59), 1600-1620.
\bibitem{Meerschaert2012}
Meerschaert, M.M., Sikorskii, A. (2012). \textit{Stochastic Models for Fractional Calculus}, De Gruyter Studies in
Mathematics, vol. 43, Walter de Gruyter Co., Berlin.
\bibitem{Meerscheart2013}
Meerschaert, M.M., Straka, P. (2013). Inverse stable subordinator. \textit{Math. Model. Nat. Phenom.} \textbf{8}(4), 1-16.
\bibitem{Norris1986}
 Norris, J.R., Rogers, L.C.G., Williams, D. (1986). Brownian motions of ellipsoids. \textit{Trans. Amer. Math. Soc.} \textbf{294}  757-765.
\bibitem{Silvestrov2004}
 Silvestrov, D.S. (2004). \textit{Limit Theorems for Randomly Stopped Stochastic Processes}. Springer.
 \bibitem{Song2022}
 Song, J., Yao, J., Yuan, W. (2022). Recent advances on eigenvalues on matrix-valued stochastic processes. \textit{J. Mult. Anal.} \textbf{188}, 104847.
 \bibitem{Song2023}
 Song, J., Xiao, Y., Yuan, W. (2023). On eigenvalues of the Brownian sheet matrix. \textit{Stoch. Proc. Appl.} \textbf{166}, 104231.
 \bibitem{Song2024}
 Song, J., Yao, J., Yuan, W. (2024). Eigenvalue distributions of high-dimensional matrix processes driven by fractional Brownian motion. \textit{Random Matrix: Theory Appl.} \textbf{13}(2), 2450009.
 \bibitem{Voiculescu1992}
 Voiculescu, D., Dykema, K., Nica, A. (1992). \textit{Free Random Variables}. CRM Monograph Series 1, American
 Mathematical Society, Providence, Rhode Island.
 \bibitem{Veillette2010}
 Veillette, M., Taqqu, M.S. (2010). Using differential equations to obtain joint moments of first-passage times of increasing L\'evy processes. \textit{Stat. Probab. Lett.} \textbf{80}, 697-705.
\bibitem{Wigner1958}
Wigner, E.P. (1958). On the distribution of the roots of certain symmetric matrices. \textit{Ann. Math.} \textbf{67}, 325-327.
\bibitem{Zhang2006}
Zhang, Z., Kwok, J.T., Yeung, D.Y. (2006). Model-based transductive learning of the kernel matrix. \textit{Mach. Learn.} \textbf{63}, 69-101.
\end{thebibliography}
\end{document}